\documentclass[12pt,reqno]{amsart}

\usepackage[T1]{fontenc}
\pdfoutput=1
\usepackage{mathptmx,microtype,graphicx,dsfont,booktabs}
\usepackage[hidelinks]{hyperref}

\usepackage{amsthm,amsmath,amssymb}
\usepackage[foot]{amsaddr}
\usepackage{enumitem}
\usepackage[nameinlink,capitalise]{cleveref}
\usepackage[font=footnotesize,margin=2cm]{caption}
\usepackage{cite}

\crefname{theorem}{Theorem}{Theorems}
\crefname{thm}{Theorem}{Theorems}
\crefname{mainthm}{Theorem}{Theorems}
\crefname{lemma}{Lemma}{Lemmas}
\crefname{lem}{Lemma}{Lemmas}
\crefname{remark}{Remark}{Remarks}
\crefname{claim}{Claim}{Claims}
\crefname{subclaim}{Sub-claim}{Sub-claims}
\crefname{prop}{Proposition}{Propositions}
\crefname{proposition}{Proposition}{Propositions}
\crefname{defn}{Definition}{Definitions}
\crefname{corollary}{Corollary}{Corollaries}
\crefname{conjecture}{Conjecture}{Conjectures}
\crefname{question}{Question}{Questions}
\crefname{chapter}{Chapter}{Chapters}
\crefname{section}{Section}{Sections}
\crefname{figure}{Figure}{Figures}
\crefname{table}{Table}{Tables}

\theoremstyle{plain}

\newtheorem{theorem}{Theorem}
\newtheorem*{theorem*}{Theorem}
\newtheorem{lemma}[theorem]{Lemma}

\newtheorem{corollary}[theorem]{Corollary}
\newtheorem{proposition}[theorem]{Proposition}

\theoremstyle{definition}
\newtheorem{definition}[theorem]{Definition}
\newtheorem{example}[theorem]{Example}
\newtheorem{problem}[theorem]{Problem}

\theoremstyle{remark}

\numberwithin{theorem}{section}
\numberwithin{equation}{section}

\newcommand{\eps}{\varepsilon}
\newcommand{\C}{{\mathcal C}}

\newcommand{\RR}{{\mathbb R}}
\newcommand{\ZZ}{{\mathbb Z}}
\newcommand{\EE}{{\mathbb E}}
\newcommand{\PP}{{\mathbb P}}
\newcommand{\G}{{\mathcal G}}
\newcommand{\K}{{\mathcal K}}

\newcommand{\E}{{\mathcal E}}
\renewcommand{\H}{{\mathcal H}}
\renewcommand{\L}{{\mathcal L}}
\newcommand{\1}{{\bf 1}}

\newcommand{\Birk}{\operatorname{Birk}}

\renewcommand{\le}{\leqslant}
\renewcommand{\ge}{\geqslant}

\makeatletter
\def\@tocline#1#2#3#4#5#6#7{\relax
  \ifnum #1>\c@tocdepth 
  \else
    \par \addpenalty\@secpenalty\addvspace{#2}%
    \begingroup \hyphenpenalty\@M
    \@ifempty{#4}{%
      \@tempdima\csname r@tocindent\number#1\endcsname\relax
    }{%
      \@tempdima#4\relax
    }%
    \parindent\z@ \leftskip#3\relax \advance\leftskip\@tempdima\relax
    \rightskip\@pnumwidth plus4em \parfillskip-\@pnumwidth
    #5\leavevmode\hskip-\@tempdima
      \ifcase #1
       \or\or \hskip 1em \or \hskip 2em \else \hskip 3em \fi%
      #6\nobreak\relax
    \dotfill\hbox to\@pnumwidth{\@tocpagenum{#7}}\par
    \nobreak
    \endgroup
  \fi}
\makeatother

\keywords{Bradley--Terry model;
Coxeter permutahedra;
digraph;
Havel–Hakimi algorithm;
majorization;
paired comparisons; 
permutahedron;
root system;
score sequence; 
signed graph;
submodular function; 
tournament;
zonotope}
\subjclass[2010]{05C20;	
11P21;	
17B22;	
20F55;	
52B05;	
62J15}	

\author[B. Kolesnik]{Brett Kolesnik}
\address{Department of Statistics, University of Warwick}
\email{brett.kolesnik@warwick.ac.uk}

\author[M. Sanchez]{Mario Sanchez}
\address{Department of Mathematics, Cornell University}
\email{mariosanchez@cornell.edu}

\begin{document}

\title
[Coxeter tournaments]
{Coxeter tournaments}

\begin{abstract}
We describe the Coxeter permutahedra, recently studied
by Ardila, Castillo, Eur and Postnikov, in terms of random Coxeter
tournaments, 
which involve cooperative and solitaire games, as well
as the usual competitive games in graph tournaments. 
In this way, we
establish a Coxeter version of Moon's theorem on random tournaments.
We present a geometric proof by the Mirsky--Thompson generalized
Birkhoff's theorem, 
a probabilistic proof by Strassen's coupling theorem,
and an algorithmic proof by a Coxeter analogue of the 
Havel--Hakimi
algorithm. 
These proofs have interpretations 
in terms of players choosing
competitors/collaborators with respect 
to relative weakness/strength.
We also introduce a natural Coxeter analogue  
of the 
Bradley--Terry model, 
from the statistical theory of paired comparisons. 
\end{abstract}

\maketitle

\tableofcontents

\section{Introduction}\label{S_intro}

{\it Coxeter combinatorics},
named after H.\ S.\ M.\ Coxeter,  
is motivated  by
the observation that a large variety of combinatorial objects, 
such as matroids, posets, graphs, and the associahedron, etc., 
are connected in many ways to the standard type $\Phi=A_{n-1}$ root system 
and its related algebra, geometry and combinatorics. 
The objective is to find combinatorial objects which generalize the usual type $A_{n-1}$ 
objects to other root systems, particularly  those 
of the types $\Phi=B_n$, $C_n$ and $D_n$. 

For instance, this program has resulted in signed graphs (see Zaslavsky \cite{Zas81,Zas82}), 
Coxeter matroids
which describe decomposition of flag varieties G/B (see Borovik, Gelfand and White \cite{BGW03}), 
parsets 
which relate to root cones in root systems (see Reiner \cite{R92}) 
and Coxeter associahedron which appear in the study of cluster algebras 
(see Hohlweg, Lange and Thomas \cite{HLT11}).

In this work, we extend the classical theory of graph tournaments to the 
Coxeter setting, via the geometric perspective developed in \cite{KS24},
and 
in relation to the Coxeter permutahedra $\Pi_\Phi$ 
studied recently by 
Ardila, Castillo, Eur and Postnikov \cite{ACEP}
(see also Kamnitzer \cite{Kam10}, where they are
called {\it pseudo-Weyl polytopes}). These polytopes
have been described geometrically in terms of 
submodular functions. 

We show that the polytopes $\Pi_\Phi$ can also 
be described in terms of {\it Coxeter tournaments}
(see \cref{S_into_cox_tour}
below), 
which are related
to orientations of signed graphs, 
as introduced by 
Zaslavsky \cite{Zas91}. 
Recall that classical tournaments 
(orientations of the complete graph $K_n$) 
involve only competitive games between players. 
As we will see, in the Coxeter setting,  
collaborative and solitaire games 
naturally arise. 

Although many of our results 
extend to more general root systems, 
we will focus on the infinite families 
of types $A_{n-1}$, $B_n$, $C_n$ and $D_n$. 
If the root system is one
of the finite exceptional 
types (as in Theorem \ref{T_KC} below) 
then games cannot, in our view,  be interpreted quite so naturally. 
We leave the details of such further extensions 
to the interested reader.

\subsection{First in a series}
This work is the first in a series; 
see also the more recent work by 
the first author, Mitchell and Przyby\l owski \cite{KMP25}
and by 
Buckland, the first author, Mitchell and Przyby\l owski \cite{BKMP25}.

Our current focus, in this work, is 
on the geometry of {\it random} Coxeter tournaments. 
Our main results, \cref{T_cox_moon_1,T_cox_moon_2,T_cox_moon_complete} below, 
establish a Coxeter analogue of 
Moon's \cite{M63} classical theorem. 
More specifically, we show that $\Pi_\Phi$ is the set of all possible mean score
sequences of random Coxeter tournaments. 
We give three proofs
of \cref{T_cox_moon_complete} (geometric, probabilistic and algorithmic) in the case of the complete $\Phi$-graphs 
$\G=\K_\Phi$ (as defined in \cref{S_into_cox_moon_comp}
below). 
Since, in the classical (type $A_{n-1}$) setting, 
tournaments are orientations
of $K_n$, 
the cases 
$\G=\K_\Phi$ will be our primary focus.

The next work in this series \cite{KMP25}
studies the combinatorics of 
{\it deterministic} Coxeter tournaments, 
establishing an analogue of Landau's \cite{Lan3} classical theorem, 
answering \cref{OP}(1) below in the case $\G=\K_\Phi$. 
The third work \cite{BKMP25} studies 
the mixing time of 
random walks 
on the sets of Coxeter tournaments 
with given score sequence; or, more specifically, 
on Coxeter analogues of the {\it tournament interchange
graphs} in  
Brualdi and Li \cite{BL84}.

\subsection{Outline}
In \cref{S_Context}, we discuss (classical and Coxeter) tournaments
and their geometry. 
The classical permutahedron and the more recent 
Coxeter generalized permutahedra are discussed in 
\cref{S_into_perm,S_into_cox_perm}, with further background
(on root systems, signed graphs, etc.)\
given in \cref{S_back}. Classical 
tournaments and their connection to 
classical permutahedra are discussed in \cref{S_into_tour}.
Coxeter tournaments are introduced in \cref{S_into_cox_tour}.

Our results 
are presented in \cref{S_Results}. 
The Coxeter analogue of Moon's theorem is stated in full generality in 
\cref{S_into_cox_moon}. 
In the complete case (\cref{S_into_cox_moon_comp}), 
our results are simpler to state, and more methods of proof are available. 
We present a geometric (\cref{S_geo}), probabilistic (\cref{S_prob}) and algorithmic (\cref{S_alg}) proof. 
An extension of the Bradley--Terry model, 
from the statistical theory of paired comparisons, 
to the Coxeter setting is discussed in 
\cref{S_PairedComparisions}. 
Deterministic Coxeter tournaments are 
discussed in \cref{S_into_cox_lan}. 

See \cref{S_proofs,S_complete} for the proofs.

\subsection{Acknowledgements}
We thank Richard Stanley 
for suggesting this line of research, 
in conversation with BK at Persi Diaconis' 
75th birthday conference. 
We also thank 
David Aldous and Federico Ardila 
for useful discussions. 
MS was partially supported by the 
National Science Foundation under Award No.\ DMS-2103391.

\section{Tournaments and polytopes}
\label{S_Context}

Throughout this work, we will as usual let $[n]=\{1,2,\ldots,n\}$. 
Similarly, we put $[-n]=-[n]=\{-1,-2,\ldots,-n\}$ and 
$[\pm n]=[n]\cup[-n]=\{\pm1,\pm 2,\ldots,\pm n\}$. 

In this section, we will give a 
brief discussion on tournaments, 
their relationship with convex geometry, 
and the Coxeter analogues studied in this work. 
Further background and 
formal definitions will be given in 
\cref{S_back} below, after
we state our main results in 
\cref{S_Results}.

\subsection{The permutahedron}\label{S_into_perm}
The {\it permutahedron} $\Pi_{n-1}$ is a classical (type $A_{n-1}$) combinatorial object, obtained by 
taking the convex hull of  
\begin{equation}\label{E_vn}
v_n = (0,1,\ldots,n-1) 
\end{equation}
and   its permutations, that is, 
\begin{equation}\label{E_perm_conv}
\Pi_{n-1} = \operatorname{conv}\{
\sigma\cdot v_n
: \sigma \in S_n\},
\end{equation}
where $S_n$ is the symmetric group, which acts on $\RR^n$ by permuting coordinates.
The reason for the ``$n-1$'' is that $\Pi_{n-1}$ is only 
$(n-1)$-dimensional. 

The permutahedron  is also the {\it graphical zonotope} 
$Z_{K_n}$
of the complete graph $K_n$ on $[n]=\{1,2,\ldots n\}$ (see, e.g., Ziegler \cite{Z95}). 
More specifically, it is the (translated) Minkowski sum of line segments
\begin{equation}\label{E_perm_mink}
\Pi_{n-1}= v_n + \sum_{i<j} [0, e_i-e_j],
\end{equation}
where $e_i$ are the standard basis vectors. 
Note that, in this sum, we have a line segment $[0, e_i-e_j]$
for each edge $ij\in E(K_n)$. 
To geometry of $\Pi_{n-1}$ is inextricably linked with $K_n$. 
Notably, the volume of $\Pi_{n-1}$ is the number of spanning trees
$T\subset K_n$, and the number of lattice points in $\Pi_{n-1}$ is the number of 
spanning forests $F\subset K_n$; 
see Stanley \cite{Sta12}, and 
Postnikov \cite{Pos09} for further generalizations).

We recall that,  
for $x,y\in\RR^n$, we say that $x$ is {\it majorized}
by $y$ 
and write $x\preceq y$ 
if 
$\sum_{i=1}^k (x_\uparrow)_i\ge (y_\uparrow)_i$ for all $k\in[n]$ with equality when $k=n$, where $z_\uparrow$ denotes the non-decreasing rearrangement of 
$z\in\RR^n$. 
Rado \cite{R52} proved that 
the permutahedron has {\it hyperplane description}
\begin{equation}\label{E_perm_hyp}
\Pi_{n-1} = \{x\in \RR^n:x\preceq v_n\},
\end{equation} 
Intuitively, $v_n$ and its permutations are the corners
of the $\Pi_{n-1}$, and all other $x\in \RR^n$
inside the convex hull are ``less spread out.''
Indeed, the concept of majorization (see, e.g., 
the textbook by Marshall, Olkin
and Arnold \cite{MOA11}) was developed in order
compare the ``spread'' in vectors (and matrices).

\subsection{Coxeter permutahedra}\label{S_into_cox_perm}
As discussed above, the 
{\it Coxeter generalized permutahedra}
studied recently in 
Ardila, Castillo, Eur and Postnikov \cite{ACEP} 
(cf.\ Kamnitzer \cite{Kam10})
extend the class of {\it generalized permutahedra} (see  
Postnikov \cite{Pos09}) to the Coxeter setting. Recall that generalized permutahedra
are obtained by deformations of classical permutahedra 
(e.g., the permutahedra, associahedron, cyclohedron, 
Pitman--Stanley \cite{SP02} polytope, etc.)\ 
which preserve directions while relocating faces.

As discussed in \cite{ACEP}, the permutahedron $\Pi_{n-1}$ (as are a number of other
classical type $\Phi=A_{n-1}$ combinatorial objects) is related to the symmetric group $S_n$, 
see \eqref{E_perm_conv} above. However, 
more generally, if $\Phi\subset V$ is the root system corresponding to some reflection 
group $W$, then a {\it $\Phi$-permutahedron} is obtained as the convex
hull of the $W$-orbit of some $v\in V$. 
More specifically, if $\Phi$ is a root system with an associated positive system 
$\Phi^+\subset \Phi$, then the corresponding $\Phi$-permutahedron can be obtained as 
the Minkowski sum (cf.\ \eqref{E_perm_mink})
 \[
 \Pi_{\Phi} = \sum_{\alpha \in \Phi^+} [-\alpha/2, \alpha/2]. 
 \] 
Equivalently (cf.\ \eqref{E_perm_conv}), 
 \[ \Pi_{\Phi} = \operatorname{conv}\{
w\cdot \rho 
 : w \in W \},\]
where 
\[
\rho = \sum_{\alpha \in \Phi^+} \alpha/2 = \sum_{i=1}^n\lambda_i
\] 
is the sum of the fundamental weights (sometimes called
the {\it Weyl vector}).
See \cref{S_root_systems} below for definitions and more details.

In \cite{ACEP}, deformations 
of $\Phi$-permutahedra are studied. It is observed that a number
of objects of interest, such as 
weight polytopes, Coxeter matroids, root cones, etc., are examples
of such Coxeter generalized permutahedra. 

 A main result of \cite{ACEP} shows that 
hyperplane descriptions 
 of Coxeter generalized permutahedra
are described  in terms of  
 submodular functions $h$. 
 In root systems of types $B_n$ and $C_n$,  the functions $h$ 
 are bisubmodular. In type $D_n$, the functions $h$ are a 
 class of  submodular functions called disubmodular.
 Using these functions,  analogues of Rado's theorem \eqref{E_perm_hyp}
 for Coxeter generalized permutahedra 
 are derived, see \cref{thm:submodular generalized permutahedra} below.

\subsection{Tournaments}\label{S_into_tour}
A {\it tournament} on a graph $G=(V,E)$ on $V=[n]$ is an orientation of its edge set. 
Usually, when studying tournaments in combinatorics,
it is assumed that $G=K_n$, but it will allow for general $G$ (often 
simply called an {\it oriented graph} in the literature). 

Informally, 
we think of each vertex as a player. A game is played between 
each pair of vertices joined by an edge, which is then directed towards the winner. 
The associated {\it score sequence} $s\in \ZZ^n$ is the in-degree sequence, listing the total 
number of wins by each player. 

More generally, a {\it random tournament} on $G$ is a collection of probabilities 
$p_{ij}=1-p_{ji}$  for each edge $ij\in E$. In this context, we think of 
$p_{ij}$ as the probability that $i$ wins against $j$.
The {\it mean score sequence} $x\in\RR^n$ lists the expected number of wins by each player, where
each coordinate 
\begin{equation}\label{E_xiA}
x_i=\sum_{ij\in E} p_{ij}.
\end{equation}

\subsubsection{Tournaments and permutahedra}
\label{S_Tournaments_and_permutahedra}
There is a strong connection between tournaments and permutahedra. 
In classical work,  Landau \cite{Lan3} showed that 
$s\in\ZZ^n$ is the score sequence of a tournament on $G=K_n$ if and only if
$s\preceq v_n$, and so, by Rado's theorem \eqref{E_perm_hyp}, if and only if
$s$ is a lattice point of the permutahedron $\Pi_{n-1}$. 
Note that $v_n$ is the score sequence of the tournament on $K_n$ in which 
each player $i$ wins/loses all games against other players $j\neq i$
of smaller/larger index. Also recall that, informally, $s\preceq v_n$ means 
that $s$ is at most as spread out as $v_n$ and has the same total sum.

Extending Landau's result to the setting of random tournaments, 
Moon \cite{M63} showed that $x\in\RR^n$ is the mean score sequence of 
a random tournament on $K_n$ if and only if $x\preceq v_n$. 

Although Rado's work preceded that of Landau, it appears 
that the connection 
with $\Pi_{n-1}$ was not recorded in the literature 
until Stanley \cite{S80} relayed this observation of Zaslavsky. 
Indeed, many proofs by various methods 
of these classical theorems of Landau and Moon have appeared 
in the literature. 
However, 
to our knowledge, none have exploited the 
geometric perspective given by the equivalent 
descriptions \eqref{E_perm_conv}--\eqref{E_perm_hyp} of the permutahedron.

In recent work \cite{KS24}, we have extended the theorems
of Landau and Moon to all multigraphs, 
via a consolidated, geometric proof which capitalizes on 
the theory of zonotopal tilings. 

\begin{theorem}[\hspace{1sp}\cite{KS24}]
\label{T_paper1}
Let $Z_M$ be the graphical zonotope of the multigraph $M$
on vertex set $[n]$.
Then 
\begin{enumerate}
\item $s\in\ZZ^n$ is a score sequence of a tournament on $M$ if and only if $s \in \ZZ^n \cap Z_M$.
\item $x\in\RR^n$ is a mean score sequence of a random tournament
on $M$ if and only if $x\in Z_M$.
\end{enumerate}
\end{theorem}

Furthermore, this zonotopal perspective allows for a refinement of these results, 
which shows that any score sequence can be realized by a 
tournament with at most a ``forest's worth of randomness;''
see \cite{KS24} for details. 

In discussing these results, Stanley \cite{Stan20} asked 
about extending the theory of tournaments to the Coxeter setting
of types $B_n$, $C_n$ and $D_n$. 
The purpose of the current work is to introduce such a theory and describe its connections
with the Coxeter generalized permutahedra \cite{ACEP}. As it turns out, 
these geometric objects can be described in terms of the mean score sequences
of random tournaments on signed graphs (see, e.g., Zaslavsky \cite{Zas81,Zas82})
in which there are cooperative and solitaire games, in addition to the 
usual (type $A_{n-1}$) competitive games in classical tournaments.

\subsection{Coxeter tournaments}\label{S_into_cox_tour}

As discussed above, Coxeter combinatorics
is concerned with extending classical combinatorial 
objects to the Coxeter setting.
In pursuit of this, Zaslavsky \cite{Zas81,Zas82}
defined the notion of a signed graph, 
which corresponds to the Coxeter analogue 
of graphs in other root systems.
These objects have been used, e.g., 
to study subarrangements of the hyperplane arrangement associated to root systems, 
the class of graphical Coxeter matroids, and the volumes 
and Ehrhart polynomials of Coxeter permutahedra, etc. 
They have also found various concrete applications; 
see, e.g., \cite{ACEP} and references therein. 

Let $\Phi$ be a root system of type $A_{n-1}, B_n, C_n$ or $D_n$. 
A (simple) {\it signed graph} $\G$ on vertex set $[n]$ has 
sets of 
\begin{itemize} 
\item negative edges $E^-\subseteq {[n]\choose 2}$,
\item positive edges $E^+\subseteq  {[n]\choose 2}$,
\item half-edges $H \subseteq [n]$,
\item loops $L\subseteq [n]$.
\end{itemize}
Note that $E^+$ and $E^-$ are not necessarily disjoint. 

Next, we define a (signed) {\it $\Phi$-graph} 
as a certain type of signed graph
in one of the root systems 
$\Phi = A_{n-1}$, $B_{n}$, $C_{n}$ or $D_{n}$. 
If $\Phi = A_{n-1}$ then it has only negative
edges 
(that is, $E^+=H=L=\emptyset$) in which case $\G$ is 
(in bijection with) 
a classical graph.  
If $\Phi = B_n$ then it has no loops, $L=\emptyset$.
If $\Phi = C_n$ then it has no half-edges, $H=\emptyset$.
If $\Phi = D_n$ then it has neither half-edges nor loops, 
$H=L=\emptyset$.

Let us emphasize here that what we call negative/positive edges are usually
instead called positive/negative edges in the literature
(see, e.g., \cite{Zas81}). However, the above (reversed) choice of terminology 
is more natural in our current 
context, as negative/positive edges will turn out to represent
competitive/collaborative games in the corresponding tournament
in which, e.g., it will be natural to think about substracting/adding
``strengths,'' etc.

The edges in $\G$ are associated 
with a subset $S \subseteq \Phi^+$ of a positive system $\Phi^+\subset\Phi$. 
The {\it Coxeter graphical zonotope} of $\mathcal{G}$ is the zonotope 
\begin{equation}
\label{E_ZG}
Z_{\mathcal{G}} = \sum_{\alpha \in S} [-\alpha/2, \alpha/2].
\end{equation}
The associated {\it Weyl vector} is 
\begin{equation}
\label{E_rhoG}
\rho_\G=\sum_{\alpha\in S}\alpha/2.
\end{equation}

Extending the definition in \cref{S_into_tour}, we 
can very naturally 
define a {\it random Coxeter tournament} on a 
 $\Phi$-graph $\G$
as a collection of probabilities:
\begin{itemize} 
\item $p_{ij}^-=1-p_{ji}^-$ for each negative edge $ij\in E^-$, 
\item $p_{ij}^+=p_{ji}^+$ for each positive edge $ij\in E^+$, 
\item $p_{i}^h$ for each half-edge $i\in H$, 
\item $p_{i}^\ell$ for each loop $i\in L$. 
\end{itemize}
As above, we view the vertices as players 
and each edge as a certain type of game. 
Negative edges correspond to the usual 
competitive games where one player wins and the 
other loses.
Positive edges correspond to cooperative games 
where both players can win by working together.
Half-edges and loops correspond to solitaire games 
which an individual player can win on their own
(the only difference between that 
loops count for twice as many points). 
More specifically, 
in Coxeter tournaments of types $B_n$, $C_n$ and $D_n$, 
we parametrize the value of games as follows: 

\begin{itemize}
\item In a competitive game (negative edge), one player wins and the other loses $1/2$
point. 
\item In a cooperative game (positive edge), both players 
win or lose $1/2$ point. 
\item In a half-edge (resp.\ loop) solitaire game, a player 
wins or loses
$1/2$ point (resp.\ $1$ point) if in type $B_n$ (resp.\ $C_n$). 
\end{itemize}
Hence, the {\it mean score sequence} of a 
random Coxeter tournament is the sequence $x\in\RR^n$
with coordinates 
$x_i$ as in \eqref{E_xi} below. 
See \cref{E_x} below
for a concise, geometric 
description. 

It might be helpful, 
although rather informal, 
to think of a loop as a 
``collaborative solitaire'' game, so worth twice the points. 
Likewise, a half-edge could be thought
of as a competitive game against an ``external
player'' whose score is not listed in
the mean score sequence.

\section{Results}
\label{S_Results}

\subsection{Moon's theorem}
\label{S_into_cox_moon}
Moon's theorem extends to the Coxeter setting
as follows; cf.\ \cref{T_paper1}(2).

\begin{theorem}
\label{T_cox_moon_1}
Let $\Phi$ be a root system of type $B_n$, $C_n$ or $D_n$. 
Let $\G$ be a $\Phi$-graph with Coxeter graphical zonotope $Z_\G$. 
Then $x \in \RR^n$ is a mean score sequence of 
a random Coxeter tournament on 
$\mathcal{G}$ if and only if $x \in Z_{\mathcal{G}}$.
\end{theorem}

Using the theory of Coxeter generalized permutahedra \cite{ACEP} 
(see \cref{S_root_systems} below for definitions) 
 we can associate 
to $Z_{\mathcal{G}}$  
 a certain function $\mu_{\mathcal{G}}$ on the set of rays of the 
 Coxeter arrangement of $\Phi$ that gives a hyperplane description of $Z_{\mathcal{G}}$. 
More specifically, suppose that 
$\lambda_1, \ldots, \lambda_{n}$ are the fundamental weights of $\Phi$. 
Then $\mu_{\mathcal{G}}: \mathcal{L}\to \RR$ 
is the Coxeter submodular function corresponding to $\G$, where   
\begin{equation}\label{E_calL}
\mathcal{L} = \{
 w\cdot \lambda_k
: w \in W,\, k \in [n]\}
\end{equation}
is the set of fundamental weight conjugates
and $W$ is the Weyl group of $\Phi$.

Coxeter mean score sequences are equivalently classified as follows
(cf.\ \eqref{E_perm_hyp} in relation to Moon's Theorem). 

\begin{theorem}
\label{T_cox_moon_2}
Let $\Phi$ be a root system of type $B_n$, $C_n$ or $D_n$. 
Let $\G$ be a $\Phi$-graph.
Then $x \in \RR^n$ is a mean score sequence of 
a random Coxeter tournament on 
$\mathcal{G}$ if and only if 
$\langle \lambda,x\rangle \le \mu_\G(\lambda)$ for all
$\lambda \in \mathcal{L}$.
\end{theorem}

This second description gives explicit formulas 
in terms of the edges of $\mathcal{G}$, 
however, the details depend on the exact root system in use.

\subsection{The complete case}
\label{S_into_cox_moon_comp}
Next, we focus on the most natural case of interest, that of complete
$\Phi$-graphs. 

The {\it complete $\Phi$-graph} $\K_\Phi$ 
is the $\Phi$-graph which includes all possible edges. 
These graphs extend in a natural way the notion of the complete
(type $A_{n-1}$) graph $K_n$. 
If $\Phi$ is of type $B_n$, $C_n$ or $D_n$ then 
all possible positive and negative edges
$E^\pm={[n]\choose 2}$ are included. 
Furthermore, if $\Phi$ is of type $B_n$ then $H=[n]$ and $L=\emptyset$; 
if $C_n$ then $H=\emptyset$ and $L=[n]$; 
and if $D_n$ 
then $H=L=\emptyset$. 

In the complete case $\G=\K_\Phi$, we provide 
three proofs of the Coxeter analogue of Moon's theorem.
To this end, we first note (in \cref{S_Wmaj}) 
that the connection with majorization extends to the 
complete Coxeter setting, allowing for the following succinct version of Moon's theorem
(\cref{T_cox_moon_1,T_cox_moon_2}) in these cases. 

Recall (see \cite{MOA11}) that for $x,y\in\RR^n$, 
we say that $x$ is {\it weakly sub-majorized} by $y$, and write
$x\preceq_w y$ if $\sum_{i=1}^k (x^\downarrow)_i\le \sum_{i=1}^k (y^\downarrow)_i$ 
for all $k\in[n]$, 
where $z^\downarrow$ denotes the non-increasing rearrangement of $z$. 
Note that, if $x\preceq_w y$ and moreover $\sum_ix_i=\sum_iy_i$, 
then $x$ is in fact {\it majorized} by $y$, written as  $x\preceq y$, as discussed above. 

\begin{theorem}
\label{T_cox_moon_complete}
Let $\Phi$ be a root system of type $B_n$, $C_n$ or $D_n$. 
Then $x \in \RR^n$ is a mean score sequence on  
the complete signed graph $\K_\Phi$ 
if and only if $|x|=(|x_1|,\ldots,|x_n|)$ is weakly sub-majorized
by the Weyl vector 
\[
\rho_\Phi=v_n+\delta_\Phi 1_n,
\]
 where
 \[
 \delta_\Phi=
 \begin{cases}
 1/2&\Phi=B_n\\
 1&\Phi=C_n\\
 0&\Phi=D_n, 
 \end{cases}
 \]  
 $1_n=(1,\ldots,1)\in \RR^n$, 
 and $v_n$ is as in \eqref{E_vn}.
\end{theorem}

\subsubsection{Geometric proof}
\label{S_geo}
Our first proof of \cref{T_cox_moon_complete} exploits a
connection between the Coxeter (signed) analogue 
\cite{Mir59a,Tho77}
of the Birkhoff polytope (the convex body of doubly stochastic matrices) and Coxeter tournaments,
and is inspired by one of the proofs of Moon's classical theorem
in the recent work by Aldous and the first author \cite{AK18}. 
Roughly speaking, the proof follows by describing a specific tournament 
$T_\phi$ corresponding  to each signed permutation $\phi$ 
which is an extreme point 
of the Coxeter permutahedron $\Pi_{\Phi}$
(see \cref{D_Phiperm} below). 
This description is intuitive, involving the interpretation of 
players $i$ with $\phi(i)$ positive/negative as those which 
are strong/weak. For instance, it type $C_n$, in the tournament
$T_\phi$, player $i$ wins its solitaire game if $\phi(i)>0$. 
Its competitive (resp.\ collaborative) games with a player $j$
are won if $\phi(i)>\phi(j)$ (resp.\ $\phi(i)+\phi(j)>0$).

\subsubsection{Probabilistic proof}
\label{S_prob}
In \cref{S_Strassen}, we give a probabilistic  proof
of \cref{T_cox_moon_complete}
by Strassen's coupling theorem \cite{Str65}.
The construction is similar in spirit to (but more involved than) the 
``football'' proof of Moon's classical theorem in the work of Aldous
and the first author \cite{AK18}.
Weak sub-majorization $x\preceq_w y$ is 
equivalent to inequality $\mu_x\preceq_{\rm inc} \mu_y$ 
with respect to the increasing stochastic order on probability measures, 
where $\mu_x$ and $\mu_y$
are uniform discrete measures on the multisets 
$\{x_1,\ldots,x_n\}$ and $\{y_1,\ldots,y_n\}$ (not to be confused
with the Coxeter submodular functions $\mu_\G$ above). 
For instance, in the case of $C_n$, Strassen's theorem 
gives sub-probability measures $\mu_i$ on $\{1,\ldots,n\}$
with means $\sum_j j\mu_i(j)=|x_i|$. 
We then extend these to  
probability measures $\nu_i$ on $\{\pm 1,\ldots,\pm n\}$
with means $x_i$, in such a way that a Coxeter tournament with mean
score sequence $x$ can be constructed for which in all games, the players $i$ 
involved (two if competitive or cooperative and one if solitaire) score an independent
number of points, distributed as $\nu_i$.

\subsubsection{Algorithmic proof}
\label{S_alg}
Our third and final proof 
of \cref{T_cox_moon_complete}
given in \cref{S_HH} is fully constructive, 
and can be viewed as a 
natural continuous 
Coxeter analogue of the classical 
Havel--Hakimi \cite{H55,H62} algorithm. 
This algorithmic proof has an intuitive description in terms of players 
seeking out potential
competitors/collaborators with respect to their relative weakness/strength
(as a certain greedy particle sliding procedure),
and is inspired by an unpublished proof of Moon's classical 
theorem by David Aldous \cite{DA}.

\subsection{Paired comparisons}\label{S_PairedComparisions}
In the classical (type $A_{n-1}$) setting, all irreducible mean
score sequences $x$ on the complete graph 
$K_n$ can be realized by the 
Bradley--Terry \cite{BT52} model (see, e.g., Theorem 1 in \cite{AK18}). 
By ``irreducible'' we mean that $x$ is {\it strictly} majorized by $v_n$, written
$x\prec v_n$, that is, all 
$\sum_{i=1}^k (x_\uparrow)_i> {k\choose 2}$ with equality only when $k=n$. 
The Bradley--Terry model, arising from 
the statistical theory of paired comparisons (see, e.g.,
Cattelan \cite{Cat12}) was in fact
introduced and studied in the much earlier work of Zermelo \cite{Zer29}
(motivated by the problem of ranking chess players
based on incomplete information).

In \cite{AK18} it is argued that Bradley--Terry is in some sense the 
Gaussian
analogue in the context of tournaments. 
Indeed, they both arise when maximizing 
Shannon \cite{Sha48} entropy. Moreover, just as $x\prec v_n$ if and only if
$x=Av_n$ for some positive definite doubly stochastic matrix $A$ 
(see 
Chao and Wong \cite{CW92} and 
Brualdi, Hwang and Pyo \cite{BWP97}), 
Gaussian densities are those associated with positive definite 
covariance matrices. 

We say that $x$ is {\it strictly weakly sub-majorized} by $y$, and write
 $x\prec_w y$, if for all $k\in[n]$, 
 we have that $\sum_{i=1}^k (x^\downarrow)_i< \sum_{i=1}^k (y^\downarrow)_i$.
 
 \begin{theorem}\label{T_BT}
Let $\Phi$ be a root system of type $B_n$, $C_n$ or $D_n$. 
Let $|x| \prec_w\rho_\Phi$ be an irreducible mean score sequence
on the complete signed graph $K_\Phi$. Then $x$ can be realized 
by a random Coxeter tournament of Bradley--Terry form, with  
$p_{ij}^\pm =\varphi(\lambda_i\pm \lambda_j)$ 
and also $p_i^h= \varphi(\lambda_i)$ (resp.\ $p_{i}^\ell= \varphi(\lambda_i)$)
if in type $B_n$ (resp.\ $C_n$), 
where 
$\varphi(u)=e^u/(1+e^u)$
is the standard logistic function and where 
$\lambda_i\in\RR$ is the ``strength'' of player $i$. 
\end{theorem}
 
\cref{T_BT} essentially follows by maximizing the entropy
of a Coxeter tournament, 
subject to the contraints \eqref{E_xi} that each player $i$
has mean score $x_i$. 
For instance, in type $C_n$, the entropy is 
\[
\sum_{ij}[p^-_{ij}\log (p^-_{ij})+(1-p^-_{ij})\log (1-p^-_{ij})
+p^+_{ij}\log (p^+_{ij})]
+\sum_i p_{i}^\ell\log(p_i^\ell).
\]
As it turns out, the strengths $\lambda_i$ are the Lagrange
multipliers in this optimization problem. 
Note that the $\lambda_i$ depend on the mean score sequence $x$ 
through the constraints \eqref{E_xi}. 

Hence the {\it closure} of the set of all mean score sequences 
of Coxeter random tournaments of Bradley--Terry form
is the set of {\it all} possible mean score sequences 
(that is, the 
Coxeter permutahedron $\Pi_{\Phi}$). 
We omit the proof, since the details are similar to 
the proof of Theorem 1 in \cite{AK18}.

Finally, let us note that the logistic function is not the only possible  
choice above. For instance, replacing $\varphi$ with the Gaussian
cumulative distribution function leads to the 
Thurstone--Mosteller model \cite{Thu27,Mos51}. 
See, e.g., the discussion following Theorem 2.7 in Joe \cite{Joe88}
for more details.

\subsection{Landau's theorem} 
\label{S_into_cox_lan}

Recall (see \cref{T_paper1}(1) above)
that, for classical (type $A_{n-1}$) graphs $G$,
all possible integer score sequences $s\in\ZZ^n$ can be 
realized deterministically. 
That is, $s \in \ZZ^n \cap Z_G$
if and only if there is a tournament
with all $p_{ij}\in\{0,1\}$ and  
score sequence $s$. 
In this section, 
we discuss an extension of Landau's theorem for Coxeter tournaments 
that holds when the 
$\Phi$-graph $\G$ is, in a certain sense, {\it balanced}. 

Before stating our results, let us 
discuss some of the subtleties 
involved with distinguishing between 
integer and deterministic score sequence
in the Coxeter setting. 
Perhaps one might expect a classification of 
score sequences of deterministic 
Coxeter tournaments on $\mathcal{G}$ in terms of the weight lattice points in $Z_{\mathcal{G}}$. 
The first issue with this idea is that $Z_{\mathcal{G}}$ is not a lattice polytope in the root lattice. 
However, this is easily remedied by considering a translation
(cf.\ \eqref{E_ZG}) 
\begin{equation}
\label{E_ZGt}
Z_{\mathcal{G}}^{\rm tr} =Z_\G+\rho_\G= \sum_{\alpha \in S} [0,\alpha]
\end{equation}
of $Z_\G$, where $\rho_\G$ is as in \eqref{E_rhoG} above. 
Then it follows directly that $x + \rho_{\mathcal{G}}$ is an integer lattice point of 
$Z_{\mathcal{G}}^{\rm tr}$ whenever $x$ is a deterministic score sequence on 
$\mathcal{G}$. However, even with this modification, 
the converse fails in general, as the following examples show.

\begin{example}\label{ex: failure of Landau's theorem}
Consider 
$\G$ with no solitaire games and  
a competitive and 
cooperative game between a pair of players. 
There are four deterministic tournaments 
obtained by setting each of $p_{12}^\pm$ to $0$ or $1$. 
After translation by $\rho=(0,1)$, the score sequences are
$(0,2)$, $(0,0)$, $(1,1)$ and $(-1,1)$. 
However, notice that $Z_{\mathcal{G}} + \rho$ also contains the lattice point $(0,1)$, 
corresponding to the random tournament with $p_{12}^\pm =1/2$. 

As another example, consider $\G$ consisting of a single solitaire loop. 
After translation by $\rho=(1)$, the two score sequences are $(0)$ and $(2)$.
However, the random tournament with $p_{11}=1/2$ has mean
score sequence $(1)$. 
\end{example}

As these examples suggest, loops and cycles with 
an odd number of positive edges can prevent 
Landau's theorem from extending. The case
of half-edges, as we will see, is more subtle. 

As in \cite{Zas81,Zas82}, we make the 
following definition. 

\begin{definition}
\label{D_balanced}
Let $\Phi$ be a root system of type $B_n$, $C_n$ or $D_n$. 
A $\Phi$-graph $\G$ is 
{\it balanced} if it contains no half-edges, loops
and all cycles have an even number of positive edges. 
\end{definition}

In particular, if $\G$ is balanced then 
$E^-\cap E^+=\emptyset$ are disjoint. 
(Also recall that, in this work, we call negative/positive edges what are
usually referred to as positive/negative edges in the literature.
As such, in the literature, ``balanced'' is usually defined to 
mean that the product of signs along any cycle is positive.)

Part (1) in our next result shows that Landau's theorem 
extends when $\G$ is balanced. In fact, in type $B_n$, 
it is possible to add half-edges. 
Part (2) is a partial converse, which shows that 
being balanced is necessary
when there
are no half-edges. 

\begin{theorem}
\label{T_CoxLan}
Let $\Phi$ be a root system of type $B_n$, $C_n$ or $D_n$. 
Let $\G$ be a $\Phi$-graph
and let $Z_\G^{\rm tr}$ be the translated Coxeter graphical zonotope of $\mathcal{G}$
as in \eqref{E_ZGt} above.
Let $\G'$ be the subgraph of $\G$ obtained by 
removing any half-edges. 
\begin{enumerate}
\item If $\G'$ is balanced then 
$s$ is a score sequence of a deterministic Coxeter tournament on $\G$ if and only if 
$s+\rho_\G \in \ZZ^n \cap Z_\G^{\rm tr}$.
\item
On the other hand, if $\G=\G'$ is unbalanced (that is, if $\G$
has no half-edges and at least one loop or cycle with an odd
number of positive edges), 
then there
are integer vectors $t\in \ZZ^n \cap Z_\G^{\rm tr}$ which can only be realized
randomly, that is, as $t=x+\rho_\G$ for some
mean score sequence $x$ of a random 
Coxeter tournament on $\G$. 
\end{enumerate}
\end{theorem}

For a  $\Phi$-graph $\G$ of 
type $B_n$, $C_n$ or $D_n$, let 
$S_\G$ denote the 
set of deterministic score sequences on $\G$. 
Then  
$S_\G+\rho_\G\subseteq \ZZ^n \cap Z_\G^{\rm tr}$. 
By \cref{T_CoxLan}, in types $C_n$ and $D_n$, 
we have that $S_\G+\rho_\G= \ZZ^n \cap Z_\G^{\rm tr}$
if and only if $\G$ is balanced (since these types have no half-edges,
and so $\G=\G'$). 
On the other hand, in type $B_n$, it is possible 
to have $S_\G+\rho_\G= \ZZ^n \cap Z_\G^{\rm tr}$
even when $\G'$ is unbalanced. 
For instance, consider the two-player Coxeter tournament 
involving a competitive, cooperative and 
half-edge solitaire game. 

The following questions remain open.  

\begin{problem}
\label{OP}
Let $\G$ be a $\Phi$-graph of type
$B_n$, $C_n$ or $D_n$.
\begin{enumerate}
\item Describe the set $S_\G$ 
of deterministic score sequences on $\G$. 
\item Determine when $S_\G+\rho_\G = \ZZ^n \cap Z_\G^{\rm tr}$. 
\end{enumerate}
\end{problem}

The answer to (2) is ``balanced'' in types
$C_n$ and $D_n$, but in type $B_n$ the situation 
is less clear. 
We note that 
(1) in the complete
case $\G=K_\Phi$
has been answered
by the first author, Mitchell and Przyby\l owski \cite{KMP25}.

\section{Background}
\label{S_back}

In this section, we briefly discuss some background information 
used in this work. In \cref{S_root_systems} we recall 
basic facts about root systems.  We refer the reader to, e.g., 
Humphreys \cite{Hum90} for proofs and more details.
In \cref{S_signed_graphs} we discuss Zaslavsky's \cite{Zas81,Zas82} 
theory of signed graphs, which are the natural setting in which to extend
the classical theory of graph tournaments. 
Finally, in \cref{S_cox_perm} we 
discuss the 
Coxeter generalized permutahedra developed recently by 
Ardila, Castillo, Eur and Postnikov \cite{ACEP}, 
which describe the geometry of Coxeter tournaments.

\subsection{Root systems}\label{S_root_systems}

Throughout, we let $V$ be a Euclidean vector space with inner product $\langle \cdot, \cdot\rangle$. 
We usually take $V$ to be $\RR^n$ with the standard orthonormal basis. 
Any vector $v \in V$ determines an automorphism $s_v$ of $V$ given by
 \[s_v(x) = x - 2\frac{\langle v, x\rangle}{ \langle v, v \rangle} v. \]

\begin{definition} 
\label{D_rootsystem}
A {\it (crystallographic) root system} 
is a finite collection of vectors $\Phi \subset V$ such that the following properties hold: 
\begin{enumerate}
\item $\operatorname{span}(\Phi) = V$, 
\item if $\alpha \in \Phi$ then the only other multiple of $\alpha$ in $\Phi$ is $-\alpha$, 
\item $\Phi$ is closed under all automorphisms $s_{\alpha}$, $\alpha \in \Phi$, and 
\item for all $\alpha, \beta \in \Phi$, we have that 
$2\frac{\langle \alpha,\beta \rangle}{\langle \alpha, \alpha\rangle}\in\ZZ$.
\end{enumerate}
Vectors $\alpha\in\Phi$ are called {\it roots}.
\end{definition}

The direct sum of two root systems 
$\Phi_1$ and $\Phi_2$ 
is defined as 
\[
\Phi_1 \oplus \Phi_2 = \{ (\alpha, 0) : \alpha \in \Phi_1\} \cup \{(0,\beta) : \beta \in \Phi_2 \}.
\] 
A root system is {\it irreducible} if it is not the direct sum of root systems.

\begin{theorem}[Killing \cite{Kil90}, Cartan \cite{Car96}] 
\label{T_KC}
The irreducible (crystallographic) root systems are classified (up to isomorphism) 
as the infinite families $A_{n-1}$, $B_n$, $C_n$ and $D_n$ 
and the exceptional types $E_6$, $E_7$, $E_8$, $F_4$ and $G_2$. 
\end{theorem}

\begin{example}
\label{E_infinitefamilies}
 The infinite families are:
\begin{itemize}
	\item $A_{n-1} = \{ e_i - e_j : i \not = j \in [n]\}$,
	\item $B_n = \{\pm e_i \pm e_j : i \not = j \in [n]\} \cup \{\pm e_i : i \in [n]\}$,
	\item $C_n = \{\pm e_i \pm e_j : i \not = j \in [n]\} \cup \{\pm 2e_i : i \in [n]\}$,
	\item $D_n = \{\pm e_i \pm e_j : i \not = j \in [n]\}$, 
\end{itemize}
where the $e_i$ form the standard orthonormal
basis of $\RR^n$. One difficulty arises in the definition of the root system $A_{n-1}$. 
Since all of the vectors of $A_{n-1}$ lie in the subspace 
$\RR_0^n = \{x\in\RR^n : \sum_{i} x_i = 0\}$, 
we must choose $V$ to be this subspace
in order for condition (1) 
in \cref{D_rootsystem} to hold. 
\end{example}

\begin{definition}\label{def: Weyl Group}
The {\it Weyl group} $W$ of $\Phi$ is the group generated 
by the reflections $\{s_{\alpha}:\alpha \in \Phi\}$.
\end{definition}

\begin{example} 
\label{E_Weyl}
Examples of Weyl groups are:
\begin{itemize}
	\item In type $A_{n-1}$, the Weyl group is isomorphic to $S_n$.
	\item In types $B_n$ and $C_n$, 
	the Weyl groups are isomorphic to the group of signed permutations $S_n^{\pm}$ of $[n]$. 
	Recall that 
	elements of $S_n^{\pm}$ are bijections of $[\pm n]$ 
	such that $\phi(-i) = -\phi(i)$.
	\item In type $D_n$, the Weyl group is isomorphic to a subgroup of $S_n^{\pm}$ 
	consisting of the signed permutations such that $| \{ i \in [n]: \phi(i) < 0\} |$ is even.
\end{itemize}
\end{example}

\begin{definition}\label{def: positive systems}
Let $\Phi \subset V$ be a root system. 
A {\it positive system} $\Phi$ is a subset of $\Phi^+\subset\Phi$ 
with the property that there exists a linear functional $h \in V^*$ 
(where $V^*$ is the dual space of $V$) 
such that 
$h(\alpha)\neq0$ for all $\alpha\in\Phi$ and 
$h(\alpha) > 0$ for all $\alpha \in \Phi^+$.
\end{definition}

For any given root system, all choices of positive systems are equivalent 
up to the action of the Weyl group. 
For this reason, it usually suffices to consider one choice of positive system.

\begin{example} \label{ex: choice of positive system}
\label{E_possystems}
We will use the following choices of $\Phi^+$ for the 
root systems in \cref{E_infinitefamilies} are 
as follows:
\begin{itemize}
	\item $A_{n-1}^+ = \{e_i - e_j : i > j \in [n]\}$,
	\item $B_n^+ = \{e_i \pm e_j : i > j \in [n]\} \cup \{e_i : i \in [n]\}$,  
	\item $C_n^+ = \{e_i \pm e_j : i > j \in [n]\} \cup \{2e_i : i \in [n]\}$,  
	\item $D_n^+ = \{e_i \pm e_j : i > j \in [n]\}$.
\end{itemize}
\end{example}

\begin{definition} Let $\Phi \subset V$ be a root system and 
$\Phi^+\subset \Phi$ a positive system. 
The {\it simple system} $\Delta$ of $\Phi^+$ is the minimal collection of vectors 
such that every $ \alpha \in \Phi^+$ is a positive linear combination of vectors in $\Delta$. 
A {\it simple system} 
of $\Phi$ is a subset of roots which is a simple system of some positive system $\Phi^+\subset \Phi$.
\end{definition}

As with positive systems, all simple systems of $\Phi$ 
are equivalent up to the action of the Weyl group.

\begin{definition} Let $\Phi\subset V$ be a root system and 
$\Delta = \{\alpha_1, \ldots, \alpha_n\}$ a simple system of $\Phi$. 
The {\it fundamental weights} of $\Phi$ associated 
with  $\Delta$ are the elements $\lambda_1, \ldots \lambda_n$ of (the dual space) $V^*$ defined by
 \[ \langle \lambda_i, \alpha_j^{\vee} \rangle = 
\1_{i=j},
 \]
where $\alpha_j^{\vee} = \frac{2}{\langle \alpha_j,\alpha_j \rangle} \alpha_j$ 
is the {\it coroot} of $\alpha_j$. 
The {\it weight lattice} of $\Phi$ is the lattice generated by integer combinations of the fundamental weights.
\end{definition}

\begin{example}
\label{E_fundweights}
Once we identify $(\RR^n)^*$ with $\RR^n$, the fundamental weights of the infinite families are:
\begin{itemize}
\item $\overline{e_1}, \overline{e_1} + \overline{e_2}, \ldots, \overline{e_1} + \cdots + \overline{e_{n-1}}$ in $A_{n-1}$,
\item	$e_1, e_1 + e_2, \ldots, e_1 + \cdots + e_{n-1}, (e_1 + \cdots + e_n)/2$ in $B_n$,
\item $e_1, e_1 + e_2, \ldots, e_1 + \cdots + e_{n-1},e_1 + \cdots + e_n$ in $C_n$,
\item	$e_1, e_1 + e_2, \ldots, e_1 + \cdots + e_{n-2},
	(e_1 + \cdots + e_{n-1} \pm e_n)/2$ in $D_n$.
\end{itemize}
For type $A_{n-1}$, 
the $\overline{e_i}$ are the representatives of $e_i$ in the quotient 
$\RR^n/\RR 1_n$ where $1_n = (1,\ldots,1) \in \RR^n$.
Notice that the roots of type $A_{n-1}$ all lie in the subspace of $\RR^n$ where the coordinates sum to $0$. 
Hence, the span is not full dimensional. Likewise, the weights live in the dual space to $V$, 
where $V$ is the 
ambient space of the root system. In the type $A_{n-1}$ case, the dual is $\RR^n/\RR1_n$.

\end{example}

\begin{definition} 
Let $\Phi$ be a root system with fundamental weights $\lambda_1, \ldots, \lambda_n$. 
Let $\L$ as in \eqref{E_calL} above
denote the set of {\it fundamental weight conjugates}.
\end{definition}

\begin{definition}
\label{D_admissible}
We say that a subset $S \subseteq [\pm n]$ is {\it admissible} if $\{i, -i\}\not\subseteq S$ for all $i \in [n]$. 
For such a set, we let $S_+=S\cap[n]$ and 
$S_-=S\cap[-n]$. Note that $S_+\cap (-S_-)=\emptyset$.
\end{definition}

\begin{example} 
Continuing with \cref{E_fundweights},  
in these cases $\mathcal{L}$ can be understood in terms of 
certain types of subsets. For any signed subset $S \subseteq [\pm n]$ let 
\[
e_S = \sum_{i \in S_+}e_i - \sum_{-i \in S_-}e_i.
\]
\begin{itemize}
	\item For type $A_{n-1}$, the fundamental weight conjugates 
	are in bijection with proper, non-empty subsets $\emptyset\neq S\subsetneq [n]$ as follows: 	 
	\[ S \mapsto \sum_{i \in S} \overline{e_i}.\]
	 \item For type $B_n$, the fundamental weight conjugates 
	 are in bijection with admissible subsets $S\subseteq[\pm n]$
	 as follows: 
	  \[ S \mapsto 
	\begin{cases}
	e_S&|S|<n\\
	\frac{1}{2}e_S&|S|=n.
	\end{cases}  	 
	   \]
	\item For type $C_n$, the fundamental weight conjugates 
	are in bijection with admissible subsets $S\subseteq[\pm n]$
	 as follows: 
	  \[ S \mapsto e_S.\]
	 \item For type $D_n$, the fundamental weight conjugates 
	 are in bijection 
	 with admissible subsets $S\subseteq[\pm n]$ with $|S| \not= n-1$ as follows:
	 \[ S \mapsto \begin{cases}
	 	e_S & \text{$|S| \le n-2$} \\
	 	\frac{1}{2}e_S & \text{$|S| = n$}.
	 \end{cases}\]
	\end{itemize}
\end{example}

\subsection{Signed graphs}\label{S_signed_graphs}
Zaslavsky's theory of signed graphs
\cite{Zas81,Zas82} extends the classical theory of graphs
to the Coxeter setting.

\begin{definition}\label{def: signed graphs} 
A (simple) {\it signed graph} $\G$ on a vertex set $[n]$ is a tuple $\G=([n],E^-, E^+, H,L)$, 
where 
\begin{itemize}
\item $E^\pm\subseteq {[n]\choose 2}$ are sets of {\it positive} and {\it negative} 
edges, 
\item $H\subseteq [n]$ is a set of  {\it half-edges}, and 
\item $L\subseteq [n]$ is a set of  {\it loops}.
\end{itemize}
\end{definition}

We will be interested in particular bijections between 
signed graphs and positive roots in a root system.

\begin{definition}\label{D_phi_graph}
 Let $\Phi$ be a root system of type $A_{n-1}$, $B_n$, $C_n$ or $D_n$ 
 with the standard choice of positive systems in  \cref{ex: choice of positive system}. 
If a {\it $\Phi$-signed graph} $\G$ is of type $B_n$ it has no loops;
if of type $C_n$ it has no half-edges; and if of type $D_n$
it has neither half-edges nor loops. 
If $\G$ is of type $A_{n-1}$ then it has no 
positive edges, half-edges nor loops, and so is 
(in bijection with) 
a simple graph.

For a $\Phi$-signed graph $\mathcal{G}$, 
let $\Gamma$ denote the bijection between the edges 
\[
E(\G)=E^-\cup E^+\cup H\cup L
\]
and its associated subset of the 
positive system $\Phi^+\subset\Phi$
given by: 
\begin{itemize}
\item in type $A_{n-1}$,
\[
\Gamma(\G) 
= \{e_i - e_j : ij \in E^-, i>j\}, 
\]
\item in type $B_n$,
\[
\Gamma(\mathcal{G}) 
=\{e_i \pm e_j : ij \in E^\pm,i>j\}
\cup \{e_i : i \in H\},
\]
\item in type $C_n$,
\[
\Gamma(\mathcal{G}) 
=\{e_i \pm e_j : ij \in E^\pm,i>j\}
\cup \{2e_i : i \in L\}, 
\]
\item in type $D_n$,
\[
\Gamma(\mathcal{G}) 
=\{e_i \pm e_j : ij \in E^\pm,i>j\}.
\]
\end{itemize}
\end{definition}

We note here that, although the $\Phi$-graphs of the root systems
of types $B_n$ and  $C_n$ are in bijection, 
the map $\Gamma$ is different for these objects, since they are subsets of different sets 
of vectors. This difference will further manifest itself in the theory of 
tournaments on these signed graphs.

\begin{definition}
\label{D_complete_signed_graphs}
The {\it complete $\Phi$-graph} $\K_\Phi$ is the $\Phi$-graph where $\G = \Phi^+$. 
\end{definition}

In particular, note that $\K_{A_{n-1}}=K_n$ is the usual complete graph.

\subsection{Coxeter generalized permutahedra}
\label{S_cox_perm}
As discussed in \cref{S_into_cox_perm}  above, the notion of 
generalized permutahedra \cite{Pos09} has recently been extended to 
other root systems \cite{ACEP}. 

\begin{definition}\label{def: Coxeter Generalized Permutahedra}
Let $\Phi \subset V$ be a root system. A 
{\it $\Phi$-generalized permutahedron} 
is a polytope whose edge directions are parallel to roots in $\Phi$. 
\end{definition}

One of the main examples of Coxeter generalized permutahedra are the Coxeter permutahedra.

\begin{definition} 
\label{D_Phiperm}
Let $\Phi$ be a root system with a positive system $\Phi^+$. 
Then the corresponding {\it $\Phi$-permutahedron} is the Minkowski sum
 \[\Pi_{\Phi} = \sum_{\alpha \in \Phi^+} [-\alpha/2, \alpha/2]. \]
Equivalently, the $\Phi$-permutahedron is the polytope
 \[ 
 \Pi_\Phi = \operatorname{conv}\{ w\cdot \rho : w \in W  \},
 \]
where 
\[
\rho = \sum_{\alpha \in \Phi^+} \alpha/2 =  
\sum_{i=1}^n \lambda_i
\]
 is the sum of the fundamental weights
 (the {\it Weyl vector}). 
\end{definition}

Notice that in type $A_{n-1}$ the $\Phi$-permutahedron is 
equal (up to a translation) to the permutahedron $\Pi_{n-1}$ defined above in 
\eqref{E_perm_conv}--\eqref{E_perm_hyp}.

The relevant examples of a Coxeter generalized permutahedra for this paper 
are the polytopes associated to the 
$\Phi$-graphs first studied by Zaslavsky in the context of signed graphs.

\begin{definition}\label{def: Coxeter Graphic Zonotopes}
Let $\G$ be a $\Phi$-graph. 
The {\it Coxeter graphical zonotope} of $\G$ is the 
$\Phi$-generalized permutahedron $Z_\G$ given by the Minkowski sum
 \[
 Z_\G = \sum_{ \alpha \in \Gamma(\G)} [-\alpha/2, \alpha/2], 
 \]
 recalling (see \cref{D_phi_graph}) that $\Gamma$ denotes the bijection from 
 $\G$ to the positive roots of $\Phi$.
\end{definition}

Notice that the Coxeter graphical zonotope of the complete $\Phi$-graph 
$K_\Phi$ 
is the $\Phi$-permutahedron $\Pi_{\Phi}$.

An important aspect of Coxeter generalized permutahedra is that 
their hyperplane descriptions are given in terms of 
submodular functions.

\begin{definition}\label{def: Coxeter Arrangement}
For a root $\alpha \in \Phi$, let 
\[
H_{\alpha} = \{x \in V : \langle \alpha, x \rangle = 0\}
\] 
be the hyperplane defined by $\alpha$. 
The collection of hyperplanes $H_{\alpha}$ is called the {\it Coxeter arrangement}. 
It defines a simplicial fan $\Sigma_{\Phi}$. 
\end{definition}

The rays of the fan $\Sigma_{\Phi}$ are generated by the 
fundamental weight conjugates 
$\L$ as in \eqref{E_calL} above. 
Any function on these generators defines a function on the cones of 
$\Sigma_{\Phi}$ by extending it linearly on each cone. 
This gives a piecewise linear function.

\begin{definition}\label{def: submodular function} 
Let $\Phi$ be a root system and $\mathcal{L}$ denote the set of fundamental weight conjugates. 
A {\it $\Phi$-submodular function} is a function $h: \mathcal{L} \to \RR$ such that
 \[ 
 h(\lambda) + h(\lambda') \ge h(\lambda + \lambda'), 
 \]
where we consider $h$ as a piecewise linear function on the cones.
\end{definition}

This notion of submodularity can be traced back to 
Kamnitzer \cite{Kam10}; see Proposition 2.2 and the proof of Lemma A.5
therein. 

In type $A_{n-1}$ the functions $h$ corresponds to usual 
submodular functions $f$ on $[n]$ such that $f([n]) = 0$. 
On the other hand, in types $B_n$ and $C_n$ they 
correspond to bisubmodular functions,  and in type $D_n$ to a type called 
disubmodular. See Section 5.2 in \cite{ACEP} for more information.

The most important consequence for us is that 
the Coxeter submodular function $h$ 
gives a hyperplane description of its corresponding 
Coxeter generalized permutahedron.

\begin{theorem}[\hspace{1sp}\cite{ACEP}, Section 5.1]
\label{thm:submodular generalized permutahedra} 
If $h$ is a $\Phi$-submodular function, then the polytope 
\[P_h 
= \{x \in \RR^n :
\langle \lambda, x\rangle \le h(\lambda)
\mbox{ for all } \lambda \in \mathcal{L}\},
\] 
where $\L$ as in \eqref{E_calL}
is the set of fundamental weight conjugates, 
is a $\Phi$-generalized permutahedron. 
On the other hand, if $P$ is a $\Phi$-generalized permutahedron, then 
\[
h_{P}(\lambda) = \operatorname{max}_{x \in P}\{\langle \lambda, x\rangle \}
\] 
is a $\Phi$-submodular function.
Furthermore, the assignments 
$h \mapsto P_h$ and $P \mapsto h_{P}$ are inverses and thus bijections.
\end{theorem}

\begin{example} 
\label{E_compCn}
For instance, 
the submodular function $h$ associated to the type $C_n$ permutahedron 
$\Pi_{C_n}$
is the function given by
 \[h( w\cdot \lambda_k) = n + (n-1) + \cdots + (n-k + 1). \]
In order words, $x \in \Pi_{C_n}$ if and only if for any $k$ distinct indices we have
 \begin{equation}\label{E_CnUB}
|x_{i_1}| + |x_{i_2}| +\cdots+ |x_{i_k}| 
\le n + (n-1) + \cdots + (n-k + 1),
 \end{equation}
that is, if and only if $|x|$ is weakly sub-majorized by 
\[
\rho_{C_n}=v_n+1_n=(1,2,\ldots,n),
\]
as in \cref{T_cox_moon_complete}.
Note that the right hand side in \eqref{E_CnUB} is equal to 
\[
\frac{k(n-k)}{2}
+\left[\frac{k(n-k)}{2}+{k\choose 2}\right]
+k,
\]
which is the maximum number of points that any given set 
$S\subset[n]$ of $k$ players can win
in total. Indeed, there are $n-k$ competitive games (worth $1/2$ point) 
with exactly one player in $S$.
Points from competitive games between players in $S$ cancel. 
There are $n-k$
cooperative games (worth $1/2$ point) with exactly one player in $S$, 
and ${k\choose2}$ cooperative games
(worth $2\cdot (1/2)=1$ point) with both players in $S$. 
Finally, there are $k$ solitaire games (worth $1$ point) 
in $S$. 
\end{example}

\section{Coxeter tournaments}
\label{S_proofs}

In this section, we extend the theory of graph tournaments
to the setting of signed graphs. In \cref{S_moon}, we prove an analogue of 
Moon's theorem, which extends the connection between 
graphical zonotopes and tournaments established
by \cref{T_paper1} in the classical setting. 
In \cref{S_landau}, we discuss the issues with extending
Landau's theorem to the Coxeter setting. 

As discussed in \cref{S_into_cox_tour}, we 
proceed as follows. 

\begin{definition} 
\label{E_signed_random_tournament}
Let $\Phi$ be a root system of type $B_n$, $C_n$ or $D_n$. 
Let $\mathcal{G}$ be a $\Phi$-graph. 
A {\it random $\Phi$-tournament} $T$ on $\G$ is a collection of probabilities 
\begin{itemize}
\item $p^-_{ij}=1-p^-_{ji}$ for each negative edge 
$ij \in E^-$,
\item $p^+_{ij}=p^+_{ji}$ for each positive edge 
$ij \in E^+$,  
\item $p_i^h$ for each half-edge $i \in H$ (only in type $B_n$),
\item $p_i^\ell$ for each loop $i \in L$ (only in type $C_n$). 
\end{itemize}
The corresponding {\it mean score sequence} of $T$ is the vector 
$x\in \RR^n$ with coordinates  (cf.\ \eqref{E_xiA} above)
\begin{equation}\label{E_xi}
x_i
=\sum_{ij\in E^-} (p_{ij}^- -1/2) 
+\sum_{ij\in E^+} (p_{ij}^+ -1/2)
+
\begin{cases}
p_i^h-1/2&i\in H\\
2(p_i^\ell-1/2)&i\in L.
\end{cases}
\end{equation}
\end{definition}

Intuitively, $\Phi$-tournaments have three different type of games.
In types $B_n$, $C_n$ and $D_n$ each negative edge represents a competitive
game, in which one player wins and the other loses $1/2$ point. 
On the other hand, each positive edge represents a cooperative 
game in both players win or lose $1/2$ point. 
In types $B_n$ and $C_n$ there are also solitaire games
in which a player wins or loses,  $1/2$ point in type $B_n$ 
and $1$ point in type $C_n$. 

A direct calculation gives the following interpretation in terms of the 
underlying root system.
Recall that $\Gamma$ is the bijection between 
$E(\G)$ and its associated subset of 
the positive system $\Phi^+$. 

\begin{proposition} 
\label{E_x}
Let $\Phi$ be a root system of type $B_n$, $C_n$ or $D_n$. 
Let $\G$ be a $\Phi$-graph. 
Then the mean score sequence $x$ of $T$ is given by
\[
x
=\sum_{e\in E(\G)}(p_e-1/2)\Gamma(e). 
\]
\end{proposition}

\subsection{Extending Moon's theorem}
\label{S_moon}
One of the classical results in the theory of tournaments is Moon's theorem \cite{M63}, 
which classifies the set of $x\in\RR^n$ 
which are mean score sequences of tournaments on the complete graph $K_n$. 
We now give a generalization of this result in
the setting of Coxeter tournaments. 

\begin{theorem}\label{T_main}
\label{thm: Coxeter Moon's Theorem} 
Let $\Phi$ be a root system of type $B_n$, $C_n$ or $D_n$. 
Let $\mathcal{G}$ be a $\Phi$-graph. 
Then $x \in \RR^n$ is a mean score sequence of a random 
$\Phi$-tournament on $\mathcal{G}$ 
if and only if 
$\langle \lambda, x\rangle \le h_{\mathcal{G}}(\lambda)$ for all $\lambda \in \mathcal{L}$, 
where $h_\G = h_{Z_\G}$ is the $\Phi$-submodular function corresponding to the 
Coxeter graphical zonotope
$Z_\G$.
\end{theorem}

\begin{proof} Let $C_\G=[0,1]^{|\G|}$ be the $|\G|$-dimensional unit cube
indexed by $\alpha\in\G$. 
The zonotope $Z_\G$ is the image of $C_\G$ under the map $\EE:C_\G \to \RR^n$ given by
 \[
 \EE(\{p_e:e\in E(\G)\}) = \sum_{e\in E(\G)}(p_{e} - 1/2)\Gamma(e). 
 \]
The points of $C_\G$ are in bijection with random $\Phi$-tournaments 
on $\G$, 
and $\EE(\{p_{e}\})$ is precisely the mean score sequence 
of the $\Phi$-tournament  $T=\{p_{e}\}$. 
Therefore, $x\in\RR^n$ is a mean score sequence 
of a random $\Phi$-tournament on $\G$ if and only if it is in the image of 
$C_\G$ under $\EE$ (that is, in $Z_\G$). Hence, 
applying \cref{thm:submodular generalized permutahedra}, 
the result follows. 
\end{proof}

Next, we derive a ``signed'' version of this result, by 
identifying the $\Phi$-submodular function $h_\G$ corresponding to the 
Coxeter graphical zonotope
$Z_\G$. 
To do this,  recall 
(see \cref{E_fundweights}) 
that the fundamental weight conjugates $\mathcal{L}$ can be viewed as admissible 
subsets $S\subseteq [\pm n]$. Recall that (see \cref{D_admissible}) that for such a set $S$, 
we let $S_+=S\cap [n]$ and $S_-=S\cap[-n]$. 

\begin{definition}
For any admissible subset $S\subset[\pm n]$, 
let $S^{||}\subseteq [n]$ denote the set given by 
$S_+ \cup(-S_-)$.
\end{definition}

The following result is obtained by direct calculations. 
We omit the proof. 

\begin{proposition}\label{prop: h description} 
Let $\Phi$ be a root system of type $B_n$, $C_n$ or $D_n$. 
Let $\mathcal{G}$ be a $\Phi$-graph. 
For any subset $S \subseteq [n]$, let 
\begin{itemize}
\item $\E_k^\pm(S)$ denote the number positive/negative edges in $E^\pm$ 
with exactly $k\in\{1,2\}$ endpoints
in $S$.
\item  $\H(S)=|S\cap H|$ denote the number of half-edges in $S$. 
\item  $\L(S)=|S\cap L|$ denote the number of loops in $S$. 
\end{itemize}
The $\Phi$-submodular function 
$h_\G = h_{Z_\G}$ acts on admissible sets $S\subset[\pm n]$ as follows.
We have that $h_{\mathcal{G}}(S)$ is equal to 
\begin{itemize}
\item 
$\frac{1}{2} \E_1^-(S^{||})
+\frac{1}{2}\E_1^+(S^{||})
+\E_2^+(S^{||})
+\frac{1}{2}\H(S^{||})$ in $B_n$,
\item 
$\frac{1}{2} \E_1^-(S^{||})
+\frac{1}{2}\E_1^+(S^{||})
+\E_2^+(S^{||})
+\L(S^{||})$ in $C_n$,
\item 
$\frac{1}{2} \E_1^-(S^{||})
+\frac{1}{2}\E_1^+(S^{||})
+\E_2^+(S^{||})$ in $D_n$.
\end{itemize}
\end{proposition}

As a result, we obtain the following equivalent 
version of \cref{T_main} (cf.\ \cref{E_compCn}). 

\begin{corollary}
\label{C_Coxeter_Moon's_Theorem}
Let $\Phi$ be a root system of type $B_n$, $C_n$ or $D_n$. 
Let $\mathcal{G}$ be a $\Phi$-graph. 
Then $x \in \RR^n$ is a mean score sequence of a random 
$\Phi$-tournament on $\mathcal{G}$ 
if and only if 
for any subset $S = \{i_1,\ldots,i_k\} \subseteq [n]$
we have that $\sum_{j=1}^k|x_{i_j}|$ is at most 
\begin{itemize}
\item 
$\frac{1}{2} \E_1^-(S)
+\frac{1}{2}\E_1^+(S)
+\E_2^+(S)
+\frac{1}{2}\H(S)$ in $B_n$,
\item
$\frac{1}{2} \E_1^-(S)
+\frac{1}{2}\E_1^+(S)
+\E_2^+(S)
+\L(S)$ in $C_n$, 
\item
$\frac{1}{2} \E_1^-(S)
+\frac{1}{2}\E_1^+(S)
+\E_2^+(S)$ in $D_n$.
\end{itemize}
\end{corollary}

\subsection{Extending Landau's theorem}
\label{S_landau}

In this section, we 
prove \cref{T_CoxLan}, which gives
a partial description of deterministic 
score sequences in the Coxeter setting. 
Recall the definitions in 
\cref{S_into_cox_lan} above. 

We begin by proving 
\cref{T_CoxLan}(2), which states that 
if $\G$ has no half-edges, then 
Landau's theorem does not extend when $\G$ is 
unbalanced. 

\begin{proof}[Proof of \cref{T_CoxLan}(2)]
Suppose that $\G$ has no half-edges and a 
loop or cycle $\C\subset \G$ with an odd number 
of positive edges. 
Consider the Coxeter random tournament which assigns
probabilities $1/2$ on $\C$, and puts probability 1
everywhere else in $\G$. We claim that there is no deterministic
tournament with the same mean score sequence $x$. 
In fact, there is no  
such tournament whose score sequence has the same total 
sum as $x$.

To see this, first note that changing the probability of any negative 
edge in $\G$ has no effect on the total sum of the mean score sequence, since
this only shifts points between the endpoints. 
On the other hand, changing the probability of a positive edge in $\C$ 
from $1/2$ to $0$ (resp.\ $1$) decreases (resp.\ increases) 
the total sum by $1$. 
On the other hand, changing the probability of a loop or 
positive edge 
outside of $\C$ 
from $1$ to $0$ 
decreases 
the total sum by $2$. 

Since $\C$ has an odd number 
of positive edges, any reassignment 
of probabilities along $\C$ to make it deterministic 
will result in increasing or decreasing the total sum
of the mean score sequence by an odd number. 
There is no way to compensate for this by changing some  
probabilities from 1 to 0 on loops and positive edges outside of $\C$, 
since any such reassignment results in decreasing the total sum
by an even number. 
\end{proof}

The rest of this section is devoted to the proof of
\cref{T_CoxLan}(1). 
Recall that we let $\G'$ denote the subgraph 
of $\G$ 
obtained by removing any half-edges. 
We will show that if $\G'$ is balanced, then 
any integer lattice point in the translated graphical 
zonotope $t\in \ZZ^n\cap Z_\G^{\rm tr}$ can be realized as $t=s+\rho_\G$, 
for some score sequence
$s$ of a deterministic Coxeter tournament on $\G$. 

We recall that the notion of a balanced signed graph  
appears in \cref{D_balanced} above, and that there
are many other characterizations. For us, the main utility 
is the following 
result of Heller and Tompkins \cite{HT56} 
and Hoffman and Gale \cite{HG56}
regarding the incidence matrix of a signed graph. 
For a signed graph $\mathcal{G}$ on vertex set $[n]$, let $I(\mathcal{G})$ be a matrix whose 
columns are the roots corresponding to each of the edges of $\mathcal{G}$. We have only 
defined this matrix up to a reordering of the columns, but that will be immaterial for us. 
Recall that a matrix is {\it totally unimodular} if every 
(maximal) 
minor has determinant $\pm 1$.

\begin{theorem}[\hspace{1sp}\cite{HT56,HG56}]\label{thm:balanced is unimodular} 
Let $\mathcal{G}$ be a signed graph 
and $\G'$ the subgraph of $\G$ obtained by 
removing any half-edges. 
Then $\G'$ 
is balanced 
if and only if $I(\G)$ is totally unimodular.
\end{theorem}

We will prove \cref{T_CoxLan}(1)
by studying the mixed subdivisions of the Coxeter graphical zonotope
$Z_\G$. 
This is a natural 
extension of the proof given in \cite{KS24}. 

\begin{definition}

A {\it zonotopal subdivision} of a zonotope $P$ is a collection of zonotopes $\{P_i\}$ such 
that
$\bigcup_i P_i = P$ and any two zonotopes $P_i$ and $P_j$ intersect properly 
(that is,  $P_i$ and $P_j$ intersect at a face of both, or else not at all) and 
their intersection is also in the collection $\{P_i\}$. 
We call the zonotopes $P_i$ the {\it tiles} of the subdivision.
\end{definition}

The following is a classical result of Shepard. Recall that 
\[
Z(v_1,\ldots,v_k)=\sum_{i=1}^k[0,v_i]
\]
is the zonotope generated by the collection of vectors $v_1, \ldots, v_k$. 

\begin{theorem}[Shepard \cite{S74}]
Let $Z=Z(v_1, \ldots, v_k)$ be a zonotope generated by the vectors $v_1, \ldots, v_k$. 
For any subset $S$ of these vectors, let $Z_S$ denote the zonotope generated by the vectors in $S$. 
Then, there exists a zonotopal subdivision of $Z$ where the tiles are the zonotopes $Z_S$, 
where $S$ ranges over the 
linearly independent subsets of 
$\{v_1, \ldots, v_k\}$.
\end{theorem}

Stanley calculated the Ehrhart polynomial of any zonotope. In particular,
this describes the number of lattice points in a lattice zonotope.

\begin{theorem}[{Stanley \cite[p.\ 557]{S91}}]
\label{prop: integer points in zonotope}
Let $Z=Z(v_1, \ldots, v_k)$ be the zonotope generated by the integer vectors $v_1, \ldots, v_k$. 
Then the number of integer lattice points in $Z$ is
given by the sum 
 \[\sum_{S}m(S), \]
where $S$ ranges over all linearly independent subsets of $\{v_1, \ldots, v_k\}$ and $m(S)$ is the 
absolute value of the 
greatest common divisor of all minors of size $|S|$ in the matrix whose columns are the vectors in $S$.
\end{theorem}

We note that the linearly independent subsets in the statement 
above need not be maximal. 

Therefore, in particular, if the matrix with column vectors $v_1, \ldots, v_k$ is totally unimodular, 
then every tile in the zonotopal subdivision has no interior lattice points.

\begin{definition}
Let $P = P_1 + \cdots + P_k$ be the Minkowski sum
of polytopes. 
A {\it mixed cell} (or {\it Minkowski cell}) $\sum_{i} B_i$ is a Minkowski sum of polytopes, 
where the vertices of $B_i$ are contained in the vertices of $P_i$. 
A {\it mixed subdivision} of $P$ is a collection of mixed cells which cover $P$ and intersect properly 
(that is, for any two mixed cells $\sum B_i$ and $\sum B_i'$ 
the polytopes $\sum_i B_i$ and $\sum_i B_i'$ intersect at a face of both, 
or else not at all).
\end{definition}

For zonotopes, every zonotopal subdivision is a mixed subdivision, and vice-versa
(see, e.g., 
De~Loera, Rambau and Santos \cite[Lemma 9.2.10]{DRS10}).
This means that every tile of a zonotopal subdivision of $Z^{\rm tr}_{\mathcal{G}}$ 
is a Minkowski sum of the faces of the segments $[0, \alpha_i]$ where $\alpha_i$ 
are the roots that correspond to the edges of $\mathcal{G}$. 
The faces of these segments are either the points $\{0\}$ and $\{\alpha_i\}$, 
or else the entire segment
$[0, \alpha_i]$.

Let $Z_{S_1}, \ldots, Z_{S_k}$ be the tiles in a mixed subdivision of 
$Z^{\rm tr}_{\mathcal{G}}$, where each $S_i$ is a linearly independent subset of the roots 
corresponding to the edges of $\mathcal{G}$. 
Then, by the previous argument, for every $S_i$, 
there exists a partition $U_i\cup V_i \cup S_i = \Gamma(\mathcal{G})$ such that
 \[Z_{S_i} = \sum_{v \in U_i} \{0\} + \sum_{v \in V_i} \{v\}+ \sum_{v \in S_i} [0, v]. \]

We can now complete the proof of \cref{T_CoxLan}. 

\begin{proof}[Proof of \cref{T_CoxLan}(1)]
 Let $s$ be a mean score sequence of $\mathcal{G}$ and put 
 $t = s + \rho_{\mathcal{G}}$. 
By \cref{thm: Coxeter Moon's Theorem}, 
we have $s\in Z_\G$ and so $t\in Z^{\rm tr}_{\mathcal{G}}$. 
Consider a zonotopal subdivison of $Z^{\rm tr}_{\mathcal{G}}$ 
and let $Z_S$ be one of the tiles containing the point $t$. 
As noted above, 
there is a partition $U \cup V \cup S = \Gamma(\mathcal{G})$ such that
  \[Z_{S} = \sum_{v \in A} \{0\} + \sum_{v \in B} \{v\}+ \sum_{v \in S} [0, v]. \]
As such, 
$t=0 + \sum_{v \in B}v + r$, 
for some 
$r \in \sum_{v \in S}[0,v]$. 
Notice that $t$ is an integer point if and only if $r$ is an integer vector.

Note that $\sum_{v \in S} [0, v]$ is the zonotope generated by the set of 
vectors $v\in S$. Since $\mathcal{G}'$ is balanced, the matrix $I(\mathcal{G})$ is 
totally unimodular by \cref{thm:balanced is unimodular}. 
This is the same matrix that appears in the lattice point count of the zonotope, 
as in \cref{prop: integer points in zonotope}. 
Therefore, every tile, including $Z_S$, has no interior lattice points.
Altogether, this means that if $t$ is an integer vector, then $r$ is an integer vector. 
Since the only lattice points of $\sum_{\alpha \in S}[0,v]$ are the vertices, 
this means that $r = \sum_{v \in S} c_{v} v$, for some   
constants $c_{v}\in\{0,1\}$, and this completes the proof. 
 \end{proof}

Let us remark that if there are no half-edges (that is, $\G=\G'$) then 
one can alternatively prove this result by relying on the classical, type $A_{n-1}$ result.  
Recall that two signed graphs are {\it sign-switching equivalent} (see, e.g., \cite{Zas81,Zas82}) 
if and only if one graph can be obtained from the other by a sequence 
of operations which flip the sign of every edge incident to a vertex. 
For a (signed) $\Phi$-graph 
$\G$ of type $B_n$, $C_n$ or $D_n$ 
we have that $\G'$ 
is balanced if and only if it is sign-switching equivalent to 
a type $A_{n-1}$ graph. 
From the polytope perspective, 
switching the signs of edges at a vertex corresponds to a reflection of a coordinate hyperplane.
In this case
$Z_\G$ 
can be reflected across coordinate hyperplanes until it is the graphical zonotope of type $A_{n-1}$ graph.  
Since these reflections map integer lattice vectors to integer lattice vectors, 
Landau's classical theorem gives the result. 
However, the proof we have given above in fact reproves the type $A_{n-1}$ case, 
rather than relying on it, and also allows for half-edges.

\begin{example}
The root lattice for the type $C_n$ root system consists of integer vectors 
whose coordinates sum to an even number. 
One might wonder if there is an analogue of Landau's theorem for type $C_n$ graphs 
stating that deterministic score sequences correspond to integer vectors 
whose coordinates sum to an even number. 
However, this is false.
To see this, consider the Coxeter tournament consisting of two loops. 
The tournament which assigns probability $1/2$ to each loop has 
mean score sequence $(0,0)$, 
which cannot be achieved deterministically.
\end{example}

\section{The complete case}
\label{S_complete}

Classical (type $A_{n-1}$) tournaments have primarily been studied 
on the complete graph $K_n$. In this situation, one obtains 
more elegant results 
and connections with other areas 
in, e.g., combinatorics, probability and optimization. 
In this section, we generalize some of these connections to the complete $\Phi$-graph 
$\K_\Phi$ case.

\subsection{$W$-majorization}
\label{S_Wmaj}
 
As discussed in \cref{S_into_tour}, 
when $G=K_n$, 
Moon's theorem 
can be stated succinctly in the language of majorization. 
That is, an $x\in\RR^n$ is a mean score sequence 
of a random tournament on $K_n$ if and only if 
$x$ is majorized by $v_n$, written as $x\preceq v_n$. 
In this section, we note that this statement generalizes to the Coxeter setting, via
the language of {\it $G$-majorization}.

\begin{definition}[{\hspace{1sp}\cite[Section C]{MOA11}}] 
Let $G$ be a group and $V$ a representation of $G$. 
We say that $v \in V$ is {\it $G$-majorized} by $u \in V$, 
denoted by $v \preceq_G u$ if $v$ is in 
$\operatorname{conv}\{ g\cdot u: g \in G\}$, that is, 
if $v$ is the convex hull of the orbit of $u$. 
\end{definition}

When $G=W$ is the Weyl group of type $A_{n-1}$ then $W$-majorization 
is the same as the usual notion of majorization. In types $B_n, C_n,$ and $D_n$, 
direct calculations show that $W$-majorization $\preceq_W$ is the same as weak sub-majorization 
$\preceq_w$, as defined in \cref{S_into_cox_moon_comp} above. 
As a result, we obtain the following, by which 
\cref{T_cox_moon_complete} above follows. 

\begin{proposition} 
Let $\Phi$ be a root system of type $B_n$, $C_n$ or $D_n$. 
Then $x\in\RR^n$ is a mean score sequence 
of a $\Phi$-tournament on the complete $\Phi$-graph $\K_{\Phi}$
if and only if $s$ is $W$-majorized by 
the Weyl vector 
$\rho_\Phi = \sum_{ \alpha \in \Phi^+} \alpha/2=v_n+\delta_\Phi 1_n$, 
as defined in \cref{T_cox_moon_complete}. 
\end{proposition}

\begin{proof} By definition, $x$ is $W$-majorized by $\rho_\Phi$ if 
\[
x \in \operatorname{conv}\{ w\cdot \rho_\Phi:w\in W\} = \Pi_{\Phi}.
\] 
Using the hyperplane description of $\Pi_{\Phi}$ given 
by \cref{thm:submodular generalized permutahedra}, 
this holds if and only if 
$\langle \lambda, x\rangle \le h(\lambda)$ for all $\lambda \in \mathcal{L}$, 
where $h$ is the Coxeter submodular function associated to $\Pi_{\Phi}$. 
The result then follows by \cref{T_main}.
\end{proof}

\subsection{Geometric proof}
\label{S_birk}

In this section, 
inspired by one of the proofs of Moon's classical 
(type $A_{n-1}$)
theorem in \cite{AK18}, 
we prove \cref{T_cox_moon_complete}
using the Coxeter analogue of Birkhoff's theorem \cite{Bir46} (cf.\ von Neumann \cite{vonNeu53}). 
For simplicity we will prove this for the type $C_n$ root system, but our arguments 
can be adapted to types $B_n$ and $D_n$.

Recall that 
Birkhoff's theorem
 states that every 
{\it doubly stochastic} matrix (non-negative with all row and column sums equal to 1)
is a mixture of permutation matrices (0/1 matrices with exactly one 1 in each row and column).
That is, the 
{\it Birkhoff polytope} $\Birk_n$ of doubly stochastic matrices $P\in\RR^{n\times n}$ 
is the convex hull of the set ${\rm Perm}_n$ of permutation matrices of the same size. 

The proof in \cite{AK18} which we are generalizing is
probabilistic. However, the strategy can be described
combinatorially, by taking the following three steps:
\begin{enumerate}
\item First, note that a vector $x\preceq v_n$ (the conditions in Moon's theorem)
 if and only if there is a doubly stochastic matrix such that $x = P v_n$. 
This is a classical result of 
Hardy, Littlewood and P{\' o}lya \cite{HLP29} (cf.\ \cite[Section 2]{MOA11}).

\item Second, by Birkhoff's theorem, note that any such $P$ is a 
convex combination of permutation matrices in the set 
$\{M_\sigma:\sigma\in S_n\}$. 

\item Third, construct a tournament associated with each permutation $\sigma\in S_n$, 
with mean score sequence equal to $M_{\sigma} v_n$. 
\end{enumerate}
Since mean score sequences are closed under convex combinations, this gives a proof of Moon's theorem. 

We note here that in \cite{AK18}, in step (1) above, instead of appealing to \cite{HLP29}, 
Strassen's coupling theorem \cite{Str65} is used to obtain a probabilistic proof. 
In this context, majorization $\preceq$ can be viewed as inequality in the convex order
(often also denoted by $\preceq$) 
of uniform probability distributions on discrete multisets. 

Just as signed graphs are the natural setting for Coxeter tournaments, 
signed permutations play a key role in extending the proof of 
Moon's theorem to the Coxeter setting. Recall (see \cref{E_Weyl} above) that 
a signed permutation $\phi\in S_n^{\pm}$ is a bijection of $[\pm n]$ 
such that $\phi(-i) = -\phi(i)$.

\begin{definition}
For a signed permutation $\phi\in S_n^\pm$, the corresponding {\it signed permutation matrix}
$A_\phi$ is the matrix that represents the standard action of $\phi$ on $\RR^n$. That is,
its entries are 
$(A_\phi)_{ij}=\pm1$ 
if $\phi(i) = \pm j$, and 0 otherwise. 
We let ${\rm Perm}_n^\pm$ denote the set of all such matrices.
\end{definition}

\begin{definition} A matrix $A = \{a_{ij}\}\in\RR^{n\times n}$ 
is {\it absolutely doubly sub-stochastic} if
and only if its absolute value ${\rm abs}(A)=\{|a_{ij}|\}$
is {\it doubly sub-stochastic} (non-negative with 
all row and column sums at most $1$). 
We let $\Birk_n^{\pm}$ denote the {\it signed Birkhoff polytope} 
of all such matrices. 
\end{definition}

Birkhoff's theorem generalizes as follows, 
allowing us to generalize step (2). 
See 
Mirsky \cite{Mir59a} and Thompson \cite{Tho77}
(cf.\ \cite[Section 2.C]{MOA11}).

\begin{theorem}
[\hspace{1sp}{\cite[Theorem 4]{Tho77}}]
\label{thm: Coxeter Birkhoff theorem}
The signed Birkhoff polytope 
$\Birk_n^{\pm}$ of absolutely doubly sub-stochastic matrices 
is the convex hull of the set 
${\rm Perm}_n^\pm$
of signed permutation matrices.
\end{theorem}

The following fact allows us to generalize step (1).

\begin{theorem}
[\hspace{1sp}{\cite[Section 2.C.4]{MOA11}}]
\label{thm: weak majorization matrix}
Let $x,y\in\RR^n$ and suppose that $y$ is non-negative. 
Then $|x|\preceq_w y$ if and only if $|x|=Sy$ for some doubly sub-stochastic matrix $S$,
in which case $x=Ay$ for some absolutely doubly sub-stochastic matrix $A$. 
\end{theorem}

By these results, we obtain the following. 

\begin{corollary} 
Let $\ell$ be the linear map from $\RR^{n\times n} \to \RR^n$ 
which sends matrices $M\mapsto M\rho_{C_n}$, 
where 
\[
\rho_{C_n}=v_n+1_n=(1,2,\ldots,n).
\]
Then the image of $\Birk_n^{\pm}$ under $\ell$ is the Coxeter permutahedron
$\Pi_{C_n}$ of type $C_n$. 
\end{corollary}

\begin{proof} 
Recall that $\Pi_{\Phi}$ is the convex hull 
of the orbit of the point $\rho_{C_n}$ under the natural action of 
$S_n^{\pm}$ on $\RR^n$.
Since $\ell$ is linear, the image of $\Birk_n^{\pm}$ under $\ell$ 
is the convex hull of the images of the vertices of $\Birk_n^{\pm}$. The image of these 
vertices is the orbit of $\rho_{C_n}$ under the action of $S_n^{\pm}$. The result follows.
\end{proof}

Finally, we generalize step (3). 

\begin{definition} Let $\phi \in S_n^{\pm}$ be a signed permutation. 
The tournament corresponding to $\phi$, denoted by $T_{\phi}$, on $\K_{C_n}$ 
is defined as follows: 

\begin{itemize}
\item for negative edges (competitive games),
 \[p_{ij}^{-}(\phi) = \begin{cases}
 1 & \text{ if $\phi(i) > \phi(j)$} \\
 0 & \text{otherwise};
 \end{cases} \]
\item for positive edges (cooperative games),
 \[p_{ij}^+(\phi) = \begin{cases}
 	1 & \text{if $\phi(i) + \phi(j) > 0$} \\
 	0 & \text{otherwise}; 
 	\end{cases}\]
\item for loops (solitaire games),
 \[ p_{i}^\ell(\phi) = \begin{cases}
  1 & \text{if $\phi(i) > 0$ } \\
  0 & \text{otherwise.}
  \end{cases} \]
\end{itemize}
\end{definition}

Naturally, we interpret $\phi(i)$
as the ``ability'' of player $i$. 
A player is ``strong/weak'' if their ability is positive/negative.
In the above tournament, competitive games are won
by the player that is more able. Strong players win 
solitaire games. Likewise, competitive games are won if 
the combined abilities of the two players equals that of a strong player.

By construction, we have the following. 

\begin{proposition}\label{prop: score sequence of signed tournament} 
Let $\phi \in S_n^{\pm}$ be a signed permutation with signed permutation matrix $A_{\phi}$. 
Then the mean score sequence of $T_{\phi}$ is $A_{\phi} \rho_{C_n}$.
\end{proposition}

\begin{proof} By construction, the mean score sequence of $T_\phi$
is 
\[
(\phi(1),\ldots,\phi(n))= A_{\phi} \rho_{C_n}.
\]
Indeed, if $\phi(i)=j>0$, then player $i$ wins its solitaire game, worth 1 point,
and all competitive and collaborative games against players $i'$
such that $|\phi(i')|<j$, for a total of $j-1$ additional points. The wins and losses
from competitive and collaborative games against all other players cancel out. 
Indeed, if some $\phi(i')>j$ (resp.\ $\phi(i')<-j$) 
then player $i$ loses/wins its competitive/collaborative (resp.\ collaborative/competitive) 
game with player $i'$.  
Hence player $i$ wins $j=\phi(i)$ points in total. 
The case that $\phi(i)=j<0$ is symmetric, and follows by similar reasoning. 
\end{proof}

We are now ready to prove the main result of this section, 
which implies \cref{T_cox_moon_complete} in the case 
that $\Phi=C_n$. 

\begin{theorem} 
\label{T_Cn}
A vector $x \in \RR^n$ is a mean score sequence of a 
$C_n$-tournament on complete graph $\K_{C_{n}}$ 
of type $C_n$ 
if and only if $|x|\preceq_w\rho_{C_n}$.
\end{theorem}

\begin{proof} It is clear that these conditions are necessary.  
On the other hand, 
suppose that $|x|\preceq_w\rho_{C_n}$. Then, 
by \cref{thm: weak majorization matrix}, 
there exists an absolutely doubly stochastic matrix $A$ such that $x = A \rho_{C_n}$. 
By \cref{thm: Coxeter Birkhoff theorem}, there exists 
numbers $\lambda_{\phi}\in[0,1]$ for all $\phi \in S_n^{\pm}$ 
summing to $\sum_{\phi \in S_n^{\pm}} \lambda_\phi = 1$ 
and so that 
$A = \sum_{\phi \in S_n^{\pm}} \lambda_{\phi} A_{\phi}$. 
Therefore by  \cref{prop: score sequence of signed tournament},
 \[
 x = \sum_{\phi \in S_n^{\pm}} \lambda_{\phi} A_{\phi} \rho_{C_n} 
 = \sum_{\phi \in S_n^{\pm}} \lambda_{\phi} x_{\phi}, 
 \]
where $x_\phi$ is the mean score sequence of the 
tournament $T_\phi$ corresponding to $\phi$.
Hence, to conclude, consider the random tournament $T_x$ with probabilities
 \[
 p_{ij}^{\pm} = \sum_{\phi \in S_n^{\pm}} \lambda_{\phi} p_{ij}^{\pm}(\phi)
 \]
and
 \[
 p_{i}^\ell= \sum_{\phi \in S_n^{\pm}} \lambda_{\phi} p_i^\ell(\phi). 
 \]
By construction, $T_x$ has mean score
 sequence $x$. 
\end{proof}

\subsection{Probabilistic proof}
\label{S_Strassen}

For $x,y\in \RR^n$, we have that $x\preceq_w y$
(see \cref{S_into_cox_moon_comp}) if and only if  
for all continuous increasing convex functions $\varphi$ we have that 
\begin{equation}\label{E_conv_order}
\sum_i \varphi(x_i)\le \sum_i \varphi(y_i).
\end{equation}
See \cite[Sections 3.C.1.b and 4.B.2]{MOA11} for a proof. 
As such, \cref{thm: weak majorization matrix}
above can be viewed as a special case of 
Strassen's coupling theorem \cite{Str65}, 
in the specific case of uniform probability distributions
on discrete multisets.
Indeed, let $\mu_x$ be uniform on $\{x_1,\ldots,x_n\}$ and 
$\mu_y$ uniform on 
$\{y_1,\ldots,y_n\}$. 
Then $\mu_x$ is bounded by $\mu_y$ in the increasing 
stochastic order, written as $\mu_x\preceq_{\rm inc} \mu_y$, 
if and only if \eqref{E_conv_order}.  
In this case, by \cite{Str65}, there is a coupling, that is, a joint distribution 
of random variables $(X,Y)$ with marginals $\mu_x$ and $\mu_y$, 
for which 
\begin{equation}\label{E_Strassen}
X\le \EE(Y|X). 
\end{equation}
Using this, we give a probabilistic proof of
\cref{T_Cn} above, similar in spirit to the 
``football'' proof of Moon's classical theorem in \cite{AK18}.
See the discussion after \eqref{E_psi} below for 
an informal sports interpretation of the probabilistic construction
given by the following proof. 

\begin{proof}[Proof of \cref{T_Cn}]
Suppose that $x\in\RR^n$ satisfies 
\[
|x|\prec_w\rho_{C_n}=(1,\ldots,n).
\]
Then, applying \eqref{E_Strassen} 
in the case that $X$ is uniform on $\{|x_1|,\ldots,|x_n|\}$
and $Y$ is uniform on $\{1,\ldots,n\}$, we obtain 
sub-probability 
measures $\mu_i$
on $[n]$ for which 
\begin{enumerate}
\item $\sum_{j=1}^n j\mu_i(j)= |x_i|$, for all $i\in[n]$; 
\item $\sum_{i=1}^n\mu_i(j)\le 1$, for all $j\in[n]$.
\end{enumerate}

The matrix $S$ with entries $s_{ij}=\mu_i(j)$ is doubly sub-stochastic.
First, we extend $S$ to a doubly stochastic matrix 
(see, e.g., von Neumann \cite{vonNeu53}) by adding some $\eps_{ij}\in[0,1]$ to each entry, 
that is, so that all rows $\sum_{j=1}^n (s_{ij}+\eps_{ij})=1$ and columns
$\sum_{i=1}^n (s_{ij}+\eps_{ij})=1$.
Then,  we define probability measures $\nu_i$ on $[\pm n]$ 
by 
\[
\nu_i(\pm j)=
\frac{\eps_{ij}}{2}+
s_{ij}
\1_{\pm x_i> 0}.
\]
In other words, for $j\in[n]$, if $x_i$ is positive/negative then we put the extra weight
$s_{ij}$ on positive/negative $j$. 
Note that if $x_i=0$ then all entries $s_{ij}=0$ in the $i$th row of $S$. 
By construction, 
\begin{itemize}
\item[(3)] $\nu_i$ has mean $x_i$, for all $i\in[n]$; 
\item[(4)] $\sum_{i=1}^n[\nu_i(j)+\nu_i(-j)]=1$, for all $j\in[n]$. 
\end{itemize}

For probability measures $\nu,\hat \nu$ let
\[
\psi^\pm(\nu,\hat \nu)
=\frac{1}{2}\PP(X\pm\hat X>0)
-\frac{1}{2}\PP(X\pm\hat X<0)
\]
and
\[
\psi^\ell(\nu)=\PP(X>0)-\PP(X<0),
\]
where $X,\hat X$ are independent random variables distributed as 
$\nu,\hat \nu$. 

We claim that 
\begin{equation}\label{E_psi}
x_i=\psi^\ell(\nu_i)+\sum_{j\neq i}[\psi^-(\nu_i,\nu_j)+\psi^+(\nu_i,\nu_j)].
\end{equation}
Note that, given this, the proof is complete, taking 
\[
p^{\pm}_{ij}=\psi^\pm(\nu_i,\nu_j)+\frac{1}{2}
\]
and
\[
p_i^\ell=\frac{\psi(\nu_i)+1}{2}.
\]

Informally speaking, each time that player $i$ is involved in a game,
they score an independent
number (possibly negative) number of points, distributed as $\nu_i$. 
Competitive games (worth $1/2$ point) are won by the player with higher score
and lost (worth $-1/2$ point) by the other player. In the case of a tie, 
no points are awarded. Likewise, cooperative games are won (worth $1/2$ point each)
by both players if their combined score is positive, and lost (worth $-1/2$ point each)
if their combined score is negative. If their combined score is 0, 
no points are awarded. Finally, in solitaire games, a player wins (worth $1$ point)
if their score is positive, loses (worth $-1$ point) if their score is negative, 
and if their score is 0 then no points are awarded. 

To verify \eqref{E_psi}, we proceed as follows. 
Let $\lambda$ be uniform on $\{-n,\ldots,n\}$. 
Since, by symmetry, all  
\[
\psi^\pm(\delta_k,\lambda)=\frac{k}{2n+1}, 
\]
it follows, by linearity and (3), that
\[
\psi^\pm(\nu_i,\lambda)=\frac{x_i}{2n+1}.
\]
Note that
\[
\psi^\pm(\nu_i,\lambda)
=\frac{1}{2n+1}\frac{1}{2}\psi^\ell(\nu_i)
+\frac{2n}{2n+1}\psi^\pm(\nu_i,\hat \lambda),
\]
where $\hat \lambda$ is uniform on $[\pm n]$.
Therefore
\begin{equation}\label{E_xi1}
x_i=\frac{1}{2}\psi^\ell(\nu_i)+2n\psi^\pm(\nu_i,\hat\lambda).
\end{equation}
By (4), we have that 
\[
\hat \lambda(\cdot)=
\frac{1}{2n}\sum_{i=1}^n[\nu_i(\cdot)+\nu_i(-\cdot)],
\]
and so, by the law of total probability, 
it follows that 
\begin{equation}\label{E_xi2}
2n\psi^\pm(\nu_i,\hat \lambda)
=
\frac{1}{2}\psi^\ell(\mu_i)+\sum_{j\neq i}\psi^-(\mu_i,\mu_j)
+\sum_{j\neq i}\psi^+(\mu_i,\mu_j).
\end{equation}
Combining \eqref{E_xi1} and \eqref{E_xi2}, we obtain 
\eqref{E_psi}, 
as required. 
\end{proof}

\subsection{Algorithmic proof}
\label{S_HH}

Finally, in this section, 
we present a constructive proof of \cref{T_cox_moon_complete},
via a recursive procedure which can be viewed as a 
continuous Coxeter analogue of the 
Havel--Hakimi \cite{H55,H62} algorithm. 
See \cref{T_exHH} for a concrete example.

\begin{table}[h!]
\footnotesize
\begin{center}

\caption{Construction of a type $C_7$ random tournament,  
with mean score sequence 
$x=(-.4,.5,2.3,3.4,-4.1,4.9,-5.2)$, via a Coxeter analogue of the 
Havel--Hakimi algorithm. In this greedy algorithm, 
players prefer to compete/cooperate with weak/strong players, 
and compete/cooperate with strong/weak players
only as necessary.}

\begin{tabular}{l|l|l|l|l|l|l|l} 

\bottomrule
$-5.2$&$4.9$&$-4.1$&$3.4$&$2.3$&$.5$&$-.4$&\\ \hline
$p_7^\ell=0$
&$p_{76}^-=0$&$p_{75}^-=.75$&$p_{74}^-=0$&$p_{73}^-=0$&$p_{72}^-=0$&$p_{71}^-=0$&\\
&$p_{76}^+=1$&$p_{75}^+=0$&$p_{74}^+=.05$&$p_{73}^+=0$&$p_{72}^+=0$&$p_{71}^+=0$&\\ \hline
$-1$&$0$&$-.25$&$-.95$&$-1$&$-1$&$-1$&$-5.2$\\ 

\bottomrule
&$3.9$&$-3.35$&$3.35$&$2.3$&$.5$&$-.4$&\\ \hline
&$q_6^\ell=0$
&$q_{65}^-=0$&$q_{64}^-=1$&$q_{63}^-=.1$&$q_{62}^-=0$&$q_{61}^-=0$&\\
&&$q_{65}^+=1$&$q_{64}^+=0$&$q_{63}^+=0$&$q_{62}^+=0$&$q_{61}^+=0$&\\ \hline
&$1$&$0$&$0$&$.9$&$1$&$1$&$3.9$\\

\bottomrule
&&$-2.35$&$2.35$&$2.2$&$.5$&$-.4$&\\ \hline
&&$p_5^\ell=0$&$p_{54}^-=0$&$p_{53}^-=0$&$p_{52}^-=0$&$p_{51}^-=.275$&\\
&&&$p_{54}^+=1$&$p_{53}^+=1$&$p_{52}^+=.375$&$p_{51}^+=0$&\\ \hline
&&$-1$&$0$&$0$&$-.625$&$-.725$&$-2.35$\\ 

\bottomrule
&&&$1.35$&$1.2$&$.125$&$-.125$&\\ \hline
&&&$q_4^\ell=0$&$q_{43}^-=1$&$q_{42}^-=.475$&$q_{41}^-=.35$&\\
&&&&$q_{43}^+=0$&$q_{42}^+=.35$&$q_{41}^+=.475$&\\ \hline
&&&$1$&$0$&$.175$&$.175$&1.35\\ 

\bottomrule
&&&&$.2$&$0$&$0$&\\ \hline
&&&&$q_3^\ell=0$&$q_{32}^-=.7$&$q_{31}^-=.7$&\\
&&&&&$q_{32}^+=.7$&$q_{31}^+=.7$&\\ \hline
&&&&$1$&$-.4$&$-.4$&.2\\

\bottomrule
&&&&&$0$&$0$&\\ \hline
&&&&&$p_2^\ell=.5$&$p_{21}^-=.5$&\\
&&&&&&$p_{21}^+=.5$&\\ \hline
&&&&&$0$&$0$&0\\ 

\bottomrule
&&&&&&$0$&\\ \hline
&&&&&&$p_1^\ell=.5$&\\ \hline
&&&&&&$0$&0
\end{tabular}\label{T_exHH}

\end{center}
\end{table}

\begin{theorem}
\label{T_Constructing_Coxeter_Tournaments}
Let $\Phi$ be a root system of type $B_n$, $C_n$ or $D_n$
and $x \in \RR^n$. 
If $|x|\preceq_w \rho_\Phi =v_n+\delta_\Phi 1_n$ 
then we can construct a random $\Phi$-tournament on the 
complete $\Phi$-graph $\K_\Phi$ with mean score sequence $x$, 
that is, probabilities such that 
\begin{equation}\label{E_si}
x_i=
\sum_{j\neq i} (p_{ij}^-+p_{ij}^+ -1) 
+
\begin{cases}
p_i^h-1/2&\Phi=B_n\\
2(p_i^\ell-1/2)&\Phi=C_n.
\end{cases}
\end{equation}
\end{theorem}

A key ingredient is the following 
result from majorization theory; 
see \cite[Section 4]{MOA11},
and the discussion therein about the results of 
Hardy, Littlewood and P{\' o}lya \cite{HLP29}, 
Karamata \cite{Kar32} and Tom{\'i}c \cite{Tom49}. 

\begin{lemma}[\hspace{1sp}{\cite[Section 4.B]{MOA11}}]
Let $x,y\in\RR^n$. We have that $x\preceq_w y$ if and only if 
\[
\sum_{i}(x_i-y_j)^+\le \sum_{i}(y_i-y_j)^+,\quad \mbox{for all }j\in[n], 
\]
where $z^+=\max\{z,0\}$.
\end{lemma}

We will use the following special case. 

\begin{lemma}
\label{L_MOA}
Let $\Phi$ be a root system of type $B_n$, $C_n$ or $D_n$
and $x \in \RR^n$. 
Then $|x|\preceq_w \rho_\Phi$ if and only if 
\[
\sum_{i}\phi_\ell (|x_i|)\le {n-(\ell-\delta_\Phi)\choose2},
\quad \mbox{for all }\ell\in\{\delta_\Phi,1+\delta_\Phi,\ldots,n-1+\delta_\Phi\}, 
\]
where $\phi_\ell(z)=(z-\ell)^+$.
\end{lemma}

Although the details of the following proof are somewhat technical, the overall idea
is rather intuitive, and boils down to 
a natural greedy algorithm. After all 
relevant quantities have been defined, we will give a detailed
informal description of the construction,
 after \eqref{E_gamma*} below. 
 
\begin{proof}[Proof of \cref{T_Constructing_Coxeter_Tournaments}]
The proof is by induction on $n$. 
Let $x\in \RR^n$ with $|x|\preceq_w (v_n+\delta_\Phi1_n)$ be given. 
For ease of exposition, and without loss of generality, we assume  that 
$|x_{1}|\le |x_{2}|\le\cdots\le |x_{n}|$. 

In the base case $n=1$ we have 
$|x_1|\le\delta_\Phi$. 
If $\Phi=D_1$, we are done, since $x_1=0$ and there are no 
probabilities to be defined. 
On the other hand, if $\Phi=B_1$ (resp.\ $\Phi=C_n$) then $|x_1|\le1/2$
(resp.\ $|x_1|\le1$). In these cases, we put 
$p_1^h=x_1+1/2$ (resp.\ $p_1^\ell=(x_1+1)/2$). 

For the inductive step, 
we describe  
a recursive algorithm that, in each step, assigns probabilities 
to all games involving the most extreme (either the weakest or strongest, 
whichever is more extreme) remaining 
player. 

Note that if $x_n=0$, then in fact all $x_j=0$. 
In this case, we can simply put all $p_{ij}^\pm=1/2$ and $p_i^h=1/2$
(resp.\  $p_i^\ell=1/2$) if in type
$B_n$ (resp.\ $C_n$). 
Hence, suppose that $x_n\neq 0$. 
We first consider the case that $x_{n}< 0$. 
The case that $x_{n}> 0$ follows by a symmetric argument
(as explained in Case 2 below). 

{\bf Case 1 ($x_{n}< 0$).} 
In this case, we find $p^\pm_{nj}$ such that 
\begin{align}\label{E_HH}
x_{n}&=\sum_{j< n} (p_{nj}^- -1/2) 
+\sum_{j< n} (p_{nj}^+ -1/2)
+
\begin{cases}
p_{n}^h-1/2&\Phi=B_n\\
2(p_{n}^\ell-1/2)&\Phi=C_n
\end{cases}\nonumber\\
&
=-(n-1+\delta_\Phi)+\sum_{j< n} (p_{nj}^-+p_{nj}^+)
+
\begin{cases}
p_{n}^h&\Phi=B_n\\
2p_{n}^\ell&\Phi=C_n
\end{cases}
\end{align}
and 
\begin{equation}\label{E_ind}
|x'|=(|x_1'|,\ldots,|x_{n-1}'|)\preceq_w (v_{n-1}+\delta_\Phi1_{n-1}), 
\end{equation}
where 
\[
x_j'=x_j+p_{nj}^- - p_{nj}^+,\quad j\neq n.
\]
Note that 
\[
x_j'+(1/2-p_{nj}^-)+(p_{nj}^+-1/2)=x_j'+p_{nj}^+-p_{nj}^-=x_j, 
\]
so (informally speaking)  $x_j'$ is the 
average number of points yet to be earned by player $j< n$, 
after winning on average $p_{nj}^+-p_{nj}^-$ points from games with player $n$. 
We also note that the $p_{nj}^\pm$ will be chosen in such a way 
 that order is preserved, that is,  so that $|x_1'|\le \cdots\le|x_{n-1}'|$. 

Since $|x|\preceq_w \rho_\Phi$, we have $|x_{n}|-\delta_\Phi\le n-1$. 
Therefore, since $x_{n}< 0$, it follows that $n-1+\delta_\Phi+x_{n}\ge0$. 

For $j< n$, let  $I_j=[|x_j|-1,|x_j|]$ be the unit interval
with right endpoint  $|x_j|$. For $\gamma\ge-1$, let 
\begin{align*}
\ell_j(\gamma)
&={\rm length}(I_j\cap [\gamma,\infty))+{\rm length}(I_j\cap [\gamma,0])\\
&=
\begin{cases}
{\rm length}([\gamma,\infty)\cap I_j)&\gamma\ge0\\
{\rm length}(I_j\cap [0,\infty))+2\cdot {\rm length}(I_j\cap [\gamma,0])&\gamma<0.
\end{cases}
\end{align*}
In other words, $\ell_j(\gamma)$ is the length of the interval to the right of 
$\gamma$, plus any such length to the left of the origin counted twice. 

Since $\delta_\Phi\le1$ and $|x_n|\ge|x_j|$ for all $j<n$, 
it follows that 
\[
\delta_\Phi-|x_{n}|\le  \sum_{j< n}(1-|x_j|)\1_{|x_j|<1}.
\]
Therefore, since $x_n<0$,  
we have that 
\[
\sum_{j< n}
\ell_j(-1)\\
=\sum_{j< n}(\1_{|x_j|\ge1}+(2-|x_j|)\1_{|x_j|<1})
\ge 
n-1+\delta_\Phi+x_n\ge0.
\]
Note that $\sum_{j< n}\ell_j(\gamma)$ is decreasing continuously 
in $\gamma\ge-1$. 
Hence 
select (the unique) such $\gamma_*\in[-1,\infty)$ for which 
\begin{equation}\label{E_gamma*}
\sum_{j< n}
\ell_j(\gamma_*)
=n-1+\delta_\Phi+x_{n}. 
\end{equation}
Using this quantity, we define the probabilities 
$p_{nj}^\pm$ as follows: 
\begin{itemize}[nosep]
\item if $x_j\le 0$, put 
$p_{nj}^-={\rm length}(I_j\cap [\gamma_*,\infty))$ 
and $p_{nj}^+={\rm length}(I_j\cap [\gamma_*,0])$,
\item  if $x_j\ge 0$, put
$p_{nj}^+={\rm length}(I_j\cap [\gamma_*,\infty))$ 
and $p_{nj}^-={\rm length}(I_j\cap [\gamma_*,0])$,
\item if $\Phi=B_n$ (resp.\ $C_n$) put $p_n^h=0$
(resp.\ $p_n^\ell=0$). 
\end{itemize}

Before continuing with the formal proof, 
let us discuss the general intuition behind our 
construction, and our proof strategy going forward:

\begin{quote}
The choice of $\gamma_*$, and the 
probabilities $p_{nj}^+$ and $p_n^h$ (or $p_n^\ell$) that it determines, 
has the following
natural interpretation in terms of a  
 greedy strategy for player $n$. 
Let us assume (Case 1 below) that player $n$ is a  
weak player, $x_n< 0$. (The other case is symmetric, 
see Case 2 below.) 
In this case, they lose/forfeit their solitaire game, $p_n^\ell=0$. 
We must then select the remaining probabilities in such a 
way that 
\[
\sum_{j< n} (p_{nj}^-+p_{nj}^+)=n-1+\delta_\Phi+x_n.
\]
To select such a $\gamma^*$, 
we think of a system of $n-1$ labelled particles on $\RR$, 
where the $j$th particle is placed at position $|x_j|$. 
Particles $j$ farther to the right
correspond to players $j$ that player $n$ would prefer to compete/cooperate with, depending on 
whether $x_j$ is negative/positive (that is, if player $j$ is weak/strong). 
Hence a natural greedy strategy for 
player $n$ is as follows. 
Imagine a 	``slider'' moving at unit rate towards the origin, initially starting
from the right of all particles. Once the slider touches a particle, 
it is ``picked up'' and slides along with it. Particles can travel for at most
a unit distance, at
which point they are ``dropped off.'' 
There are two cases to consider. 

{\it Case 1a.} If the total distance 
travelled to the left by all particles equals $n-1+\delta_\Phi+x_n$
once the slider reaches some point $\gamma_*>0$ to the right of the origin, 
then we simply let $p_{nj}^-$ (resp.\ $p_{nj}^+$) be the  distance
travelled to the left by particle $j$ if $x_j<0$ (resp.\ $x_j>0$), given by 
${\rm length}(I_j\cap [\gamma_*,\infty))$. 
Note that, in this case,  player $n$ has managed to win its required (average) number
of points by only competing/cooperating with weak/strong players. 
 
{\it Case 1b.} On the other hand, if some particles reach the origin before the total distance travelled 
reaches $n-1+\delta_\Phi+x_n$, then we modify the construction as follows. In this case, 
player $n$ is not be able to avoid competing/cooperating with some strong/weak players. 
Once the slider reaches the origin, we imagine particles at the origin simultaneously traveling 
to the left and right at the same rate
(effectively, being held in place) until they have traveled a unit distance or else the total distance 
travelled (to the left and right)
by all particles
reaches $n-1+\delta_\Phi+x_n$ (whichever comes first). 
At this point, for some $\gamma_*\in [-1,0]$, note that the $j$th particle 
will have travelled ${\rm length}(I_j\cap [\gamma_*,\infty))$  
to the left and ${\rm length}(I_j\cap [\gamma_*,0])$ to the right. 
If $x_j<0$ (resp.\ $x_j>0$) we let $p_{nj}^-$ and $p_{nj}^+$ (resp.\ $p_{nj}^+$ and $p_{nj}^-$) 
be these distances travelled 
to the left and right. Note that, in this case, player $n$ prioritizes 
competition/cooperation with weak/strong players, and cooperates/competes with such
players only as necessary. 
\end{quote}

We return to the formal proof. Note that \eqref{E_HH} holds by the choice of $\gamma_*$, 
and that by construction we have $|x_1'|\le \cdots\le |x_{n-1}'|$. 
Next, we verify \eqref{E_ind}. 
In doing so, we take cases with respect to whether 
$\gamma_*\in[-1,0]$ or $\gamma_*>0$.

{\bf Case 1a.} Suppose that $\gamma_*>0$.
Then 
\[
|x_j'|=
\begin{cases}
\max\{\gamma_*,|x_j|-1\}&|x_j|\ge\gamma_*\\
|x_j|&|x_j|<\gamma_*.
\end{cases}
\]
In this case, we appeal to \cref{L_MOA}.
Let $\ell\in\{\delta_\Phi,1+\delta_\Phi,\ldots,n-2+\delta_\Phi\}$. 
If $\ell>\gamma_*$ then 
\begin{multline*}
\sum_{j=1}^{n-1}\phi_\ell(|x_{j}'|)
=\sum_{j=1}^ {n-1}\phi_\ell(|x_{j}|-1)
=\sum_{j=1}^{n-1}\phi_{\ell+1}(|x_{j}|)\\
\le \sum_{j=1}^{n}\phi_{\ell+1}(|x_j|)
\le{n-(\ell+1-\delta_\Phi)\choose 2}
={(n-1)-(\ell-\delta_\Phi)\choose2}.
\end{multline*}
Otherwise, if  $\ell\le\gamma_*$ then by construction 
we have that 
\[
\sum_{j=1}^{n-1}[\phi_\ell(|x_{j}|)-\phi_\ell(|x_{j}'|)]
=n-1+\delta_\Phi+x_{n}.
\]
Therefore, since 
$\gamma_*\le |x_{n}|$ 
and $x_{n}\le 0$, 
it follows that 
\[
\sum_{j=1}^{n}\phi_\ell(|x_j|)-\sum_{j=1}^{n-1}\phi_\ell(|x_{j}'|)
=\phi_\ell(|x_{n}|)+n-1+\delta_\Phi+x_{n}=n-1+\delta_\Phi-\ell,
\]
and so 
\begin{multline*}
\sum_{j=1}^{n-1}\phi_\ell(|x_{j}'|)
=\sum_{j=1}^{n}\phi_\ell(|x_j|)-(n-1+\delta_\Phi-\ell)\\
\le  {n-(\ell-\delta_\Phi)\choose2}-[(n-1)-(\ell-\delta_\Phi)]
={(n-1)-(\ell-\delta_\Phi)\choose2}.
\end{multline*}
Therefore, by \cref{L_MOA}, we find that  $|x'|\preceq_w \rho_\Phi$,
as required. 

{\bf Case 1b.} On the other hand, suppose that $\gamma_*\in[-1,0]$. 
In this case, we show that $|x'|\preceq_w (v_{n-1}+\delta_\Phi 1_{n-1})$
by appealing directly to the definition of weak sub-majorization. 
Note that, in this case, we
have that  
all $|x_j'|=(|x_j|-1)^+$. 
Since $|x|\preceq_w \rho_\Phi=v_n+\delta_\Phi 1_n$, we have that, 
for any $S\subseteq [n]$ of size $k$, 
\[
\sum_{j\in S}|x_j|\le 
k\delta_\Phi+\sum_{j=1}^k(n-j)
={k\choose2}+k(n-k+\delta_\Phi). 
\] 
Therefore,  if $S\subset[n-1]$ is of size $k$, then 
\[
\sum_{j\in S}|x_j'|
= \sum_{j\in S'}(|x_j|-1)
\le {k'\choose2}+k'[(n-1)-k'+\delta_\Phi]
\]
where $S'$ of size $k'\le k$ is the set of $j\in S$ for which $|x_j|>1$.
Since the right-hand side is non-decreasing in $k'\le k$, it follows that 
\[
\sum_{i\in S}|x_i'|
\le {k\choose2}+k[(n-1)-k+\delta_\Phi].
\]
Hence $|x'|\preceq_w (v_{n-1}+\delta_\Phi 1_{n-1})$, as required. 

This concludes the proof in Case 1. 

{\bf Case 2 ($x_{n}> 0$).} 
On the other hand, if $x_n\ge0$, we can 
instead find the probabilities  $q_{nj}^\pm=1-p_{nj}^\pm$
and $q_n^h=1-p_n^h$ (or $q_n^\ell=1-p_n^\ell$)  
by a symmetric argument. 
As before, fix (the unique) $\gamma_*$ such that 
\[
\sum_{j< n}
\ell_j(\gamma_*)
=n-1+\delta_\Phi-x_{n}.
\]
Then define probabilities as follows:  
\begin{itemize}[nosep]
\item if $x_j\le 0$, put 
$q_{nj}^+={\rm length}(I_j\cap [\gamma_*,\infty))$ 
and $q_{nj}^-={\rm length}(I_j\cap [\gamma_*,0])$,
\item  if $x_j\ge 0$, put 
$q_{nj}^-={\rm length}(I_j\cap [\gamma_*,\infty))$ 
and $q_{nj}^+={\rm length}(I_j\cap [\gamma_*,0])$,
\item if $\Phi=B_n$ (resp.\ $C_n$) put $q_n^h=0$
(resp.\ $q_n^\ell=0$).
\end{itemize}
By construction, we have that 
\begin{equation}\label{E_HH'}
\sum_{j< n} (q_{nj}^-+q_{nj}^+)
=n-1+\delta_\Phi-x_n.
\end{equation}
Arguing as in Case 1, it can be shown that 
\begin{equation}\label{E_ind'}
|x'|=(|x_{1}'|,\ldots,|x_{n-1}'|)\preceq_w (v_{n-1}+\delta_\Phi1_{n-1}), 
\end{equation}
where 
\[
x_j'=x_j-q_{nj}^- + q_{nj}^+,\quad j< n.
\]

This concludes the proof in Case 2. 

To finish the proof, we note that all 
probabilities $p_{ij}^\pm$ and $p_i^h$ (or $p_h^\ell$) can be 
defined recursively by the above procedure, 
beginning with $i=n$.
\end{proof}


\makeatletter
\renewcommand\@biblabel[1]{#1.}
\makeatother

\providecommand{\bysame}{\leavevmode\hbox to3em{\hrulefill}\thinspace}
\providecommand{\MR}{\relax\ifhmode\unskip\space\fi MR }
\providecommand{\MRhref}[2]{%
  \href{http://www.ams.org/mathscinet-getitem?mr=#1}{#2}
}
\providecommand{\href}[2]{#2}

\end{document}